\documentclass[10pt,reqno]{amsart}
\usepackage[backref=page]{hyperref}

\usepackage{graphicx}
\usepackage{cite}
\usepackage{amssymb}
\usepackage{amsfonts}
\usepackage[]{amsmath}
\usepackage[]{epsfig}
\usepackage{mathtools}
\usepackage[]{float}
\usepackage{setspace}
\usepackage{tikz}

\newtheorem{remark}{Remark}[section]

\renewcommand{\Im}{\operatorname{Im}}
\renewcommand{\Re}{\operatorname{Re}}
\numberwithin{equation}{section}
\newcommand{\R}{\mathbb R}

\usepackage[english, activeacute]{babel}
\usepackage{amsmath,amsthm,amsxtra}
\usepackage{epsfig}
\usepackage{amssymb}
\usepackage{latexsym}
\usepackage{amsfonts}
\usepackage{hyperref}
\usepackage{pdftricks}
\newtheorem{thm}{Theorem}[section]
\newtheorem{cor}[thm]{Corollary}
\newtheorem{lemma}[thm]{Lemma}
\newtheorem{lem}[thm]{Lemma}
\newtheorem{prop}[thm]{Proposition}

\theoremstyle{remark}

\numberwithin{equation}{section}

\newcommand{\la}{\lambda}

\newcommand{\al}{\alpha}

\def\bm{\left( \begin{array}{cc}}
\def\endm{\end{array}\right)}

\newcommand{\be}{\begin{equation}}
\newcommand{\ee}{\end{equation}}
\newcommand{\ba}{\left(\begin{array}{c}}
\newcommand{\ea}{\end{array}\right)}
\newcommand{\bea}{\begin{eqnarray}}
\newcommand{\eea}{\end{eqnarray}}
\newcommand{\bee}{\begin{eqnarray*}}
\newcommand{\eee}{\end{eqnarray*}}
\newcommand{\ben}{\begin{enumerate}}
\newcommand{\een}{\end{enumerate}}

\newcommand{\grada}{\langle \sqrt{\H_a} \rangle}

\newcommand{\ep}{\epsilon}

\newcommand{\qtq}[1]{\quad\text{#1}\quad}
\renewcommand{\H}{\mathcal{H}}

\begin{document}
\pagenumbering{arabic}	
\title[Threshold dynamics]{Threshold dynamics for the $3d$ radial cubic NLS with repulsive inverse-square potential}

\author[L. Baker]{Luke Baker}
\address{Department of Mathematics, University of Oregon}
\email{lukebake@uoregon.edu}

\author[L. Campos]{Luccas Campos}
\address{Department of Mathematics, Federal University of Minas Gerais}
\email{luccascampos@gmail.com}

\author[J. Murphy]{Jason Murphy}
\address{Department of Mathematics, University of Oregon}
\email{jamu@uoregon.edu}

\author[R. Scarpelli]{Renzo Scarpelli}
\address{Department of Mathematics, Federal University of Minas Gerais}
\email{renzoscb@ufmg.br}

\begin{abstract} We classify the dynamics of radial $H^1$ solutions at the (radial) ground state threshold for the $3d$ cubic NLS in the presence of a repulsive inverse-square potential. 
\end{abstract}

\maketitle

\section{Introduction}

We consider the three-dimensional cubic nonlinear Schr\"odinger equation in the presence of an inverse-square potential.  The equation takes the form
\begin{equation}\label{NLSa} 
\begin{cases}
i\partial_tu - \H_a u =- |u|^2u, \\
u|_{t=0}\in H^1(\R^3),
\end{cases}
\end{equation}
where  $\H_a := -\Delta+a|x|^{-2}$ and we impose $a>-\tfrac14$ in order to guarantee the positivity of this operator.  The Schr\"odinger operator $\H_a$ arises in a variety of physical settings and as an important scaling limit (see e.g. \cite{Burq2003, Kalf1975}). This model has become an important case study in recent years as a non-perturbative generalization of the standard cubic NLS.  It has the interesting feature of breaking the space-translation symmetry while preserving the scaling symmetry.  

Our contribution in this work is to characterize the dynamics of radial solutions at the radial ground state threshold in the setting of a \emph{repulsive} inverse-square potential (i.e. with $a>0$).  We state our main result as Theorem~\ref{T} below after some discussion of equation \eqref{NLSa} and several related results.

Equation \eqref{NLSa} is the Hamiltonian flow for the following \emph{energy}:
\[
E_a(u) = \tfrac12\| u\|_{\dot H_a^1}^2 - \tfrac14\|u\|_{L^4}^4,\quad \|u\|_{\dot H_a^1}:=\sqrt{\langle u,\H_a u\rangle}. 
\]
In particular, $E_a(u)$ is conserved along the flow for \eqref{NLSa}.  Additionally, the gauge-invariance of the nonlinearity guarantees the conservation of \emph{mass}, defined by
\[
M(u) = \|u\|_{L^2}^2. 
\]
We note that in the case $a=0$, equation \eqref{NLSa} reduces to the standard cubic NLS
\[
i\partial_t u + \Delta u = -|u|^2 u,\quad u|_{t=0}\in H^1(\R^3). 
\]

For $a\in(-\tfrac14,0]$, there exists a unique, nonnegative, radial solution (the \emph{ground state}) $Q_a$ to the nonlinear elliptic equation
\[
-Q_a - \mathcal{H}_a Q_a = - Q_a^3,
\]
which can be constructed as an optimizer for a Gagliardo--Nirenberg inequality adapted to $\mathcal{H}_a$ \cite{KVMZ, yang2026uniqueness}.  For $a>0$ optimizers do not exist in $H^1$, although one can restore compactness and construct an optimizer if one restricts to radial functions (cf. \cite{KVMZ}).  In this case we continue to denote the optimizer by $Q_a$. 

The ground state plays a central role in the classification of dynamics for solutions to \eqref{NLSa}.  For solutions below the mass-energy of the ground state, there is a \emph{scattering/blowup dichotomy}.  We recall that \emph{scattering} refers to asymptotically linear dynamics, which in the case of \eqref{NLSa} is equivalent to finiteness of the $L_{t,x}^5$-norm on $\R\times\R^3$.  Writing $x\wedge y$ for $\min\{x,y\}$, one has the following result: 

\begin{thm}[Dynamics below the threshold, \cite{HolmerRoudenko2008, DuyckaertsHolmerRoudenko2008,KVMZ}]\label{T:KMVZ} Let $a>-\tfrac14$ and define
\[
\mathcal{E}_a=M(Q_{a})E_{a}(Q_{a}) \qtq{and} \mathcal{K}_a = \|Q_{a}\|_{L^2}\|Q_{a}\|_{\dot H^1_{a}}.\footnote{The values $\mathcal{E}_a$ and $\mathcal{K}_a$ can be defined purely in terms of the sharp constant in the Gagliardo--Nirenberg inequality.  In particular, Theorem~\ref{T:KMVZ} does not depend on the uniqueness of the ground state. For the problem of classifying threshold dynamics, on the other hand, the uniqueness of the ground state plays an important role.}
\]
Let $u_0\in H^1(\R^3)$ satisfy $M(u_0)E_a(u_0)<\mathcal{E}_{a\wedge 0}$.
\begin{itemize}
\item[(i)] If $\|u_0\|_{L^2}\|u_0\|_{\dot H_a^1} < \mathcal{K}_{a\wedge 0}$, then the solution to \eqref{NLSa} with initial data $u_0$ is global and scatters.
\item[(ii)] If $\|u_0\|_{L^2}\|u_0\|_{\dot H_a^1} > \mathcal{K}_{a\wedge 0}$ and $u_0$ is radial or $xu_0\in L^2$, then the solution to \eqref{NLSa} with initial data $u_0$ blows up in finite time in both time directions.
\end{itemize}

For $a>0$, suppose $u_0\in H_{\text{rad}}^1$ satisfies $M(u_0)E_a(u_0)<\mathcal{E}_a$. Then the analogues of (i) and (ii) hold with $\mathcal{K}_{a\wedge 0}$ replaced with $\mathcal{K}_a$
 \end{thm}

This result was established first for the standard cubic NLS \cite{HolmerRoudenko2008, DuyckaertsHolmerRoudenko2008}; the case of the inverse-square potential was addressed later in the work \cite{KVMZ}.  Note that for $a\leq 0$, the ground state solution $e^{it}Q_a$ is an example of a global nonscattering solution.  Thus it is clear that $M(Q_a)E_a(Q_a)$ is the correct mass-energy threshold for obtaining a simple scattering/blowup dichotomy.  Theorem~\ref{T:KMVZ} further asserts that for $a>0$ the correct threshold for such a dichotomy (for general $H^1$ data) is $M(Q_0)E_0(Q_0)$\footnote{In fact, by considering data of the form $[1-\tfrac{1}{n}]Q_0(\cdot-x_n)$ with $|x_n|\to\infty$, one can produce (nonradial) solutions $u_n$ satisfying $M(u_n)E_a(u_n)\nearrow \mathcal{E}_0$ and $\|u_n(0)\|_{L^2}\|u_n(0)\|_{\dot H^1}\nearrow \mathcal{K}_0$ such that $\|u_n\|_{L_{t,x}^5}\to\infty$ (see \cite[Theorem~1.5]{KVMZ}).}\label{footnoteref}. On the other hand, if one restricts to \emph{radial} solutions, one obtains the larger mass-energy threshold $M(Q_a)E_a(Q_a)$. 

With Theorem~\ref{T:KMVZ} in place, it is natural to investigate the dynamics of solutions living exactly at the ground state threshold.  In this case, new dynamics are possible.  In particular, in addition to the ground state solution itself, there exist heteroclinic orbits connecting the ground state to scattering/blowup.  Using these special solutions, the possible solution behaviors were characterized in \cite{DR, yang2026uniqueness} in the regime $a\leq 0$:

\begin{thm}[Threshold dynamics, $a\leq 0$, \cite{DR, yang2026uniqueness}]\label{T:dynamics1} Let $a\in(-\tfrac14+\frac{1}{9},0]$.  

\textbf{Heteroclinic orbits.} There exist radial solutions $Q_a^\pm$ to \eqref{NLSa} defined on $I^\pm\supset[0,\infty)$ with $M(Q_a^\pm)E_a(Q_a^\pm)=\mathcal{E}_a$ that converge exponentially to $Q_a$ in $H^1$ as $t\to\infty$.  \begin{itemize}
\item $Q_a^-$  satisfies $\|Q_a^-(0)\|_{\dot H_a^1}<\|Q_a\|_{\dot H_a^1}$ and scatters backward in time.  
\item  $Q_a^+$ satisfies $\|Q_a^+(0)\|_{\dot H_a^1}>\|Q_a\|_{\dot H_a^1}$ and blows up in finite negative time. 
\end{itemize}

\textbf{Classification of dynamics.} Suppose $u_0\in H^1$ satisfies $M(u_0)E_a(u_0)=\mathcal{E}_a$, and let $u$ be the solution to \eqref{NLSa} with $u|_{t=0}=u_0$. 
\begin{itemize}
\item[(i)] If $\|u_0\|_{L^2}\|u_0\|_{\dot{H}^1_a} < \mathcal{K}_a$, then $u$ either scatters in both time directions or $u = Q_a^-$ up to symmetries\footnote{When we say two solutions $u$ and $v$ agree \emph{up to symmetries}, we mean that there exist $\theta\in\R$, $\lambda\in(0,\infty)$, and $t_0\in\R$ such that
\[
u(t,x) = e^{i\theta}\lambda v(\lambda^2(t-t_0),\lambda x)\qtq{or} u(t,x) = e^{i\theta}\lambda \bar v(\lambda^2(t_0-t),\lambda x).
\]
 In the case $a=0$, we additionally include a spatial translation parameter in the formulas above. 
}.
\item[(ii)] If $\|u_0\|_{L^2}\|u_0\|_{\dot{H}^1_a} = \mathcal{K}_a$, then $u(t) = e^{it}Q_a$ up to symmetries. 
\item[(iii)] If $\|u_0\|_{L^2}\|u_0\|_{\dot{H}^1_a} >\mathcal{K}_a$ and $u_0$ is radial or $xu_0 \in L^2(\mathbb{R}^3)$, then $u$ either blows up in finite time in both time directions or $u = Q_a^+$ up to symmetries.
\end{itemize}
\end{thm}
This result was also established first for the standard cubic NLS \cite{DR}, with the case of an attractive inverse-square potential being treated quite recently in the work \cite{yang2026uniqueness}.

Note that Theorem~\ref{T:dynamics1} does {not} address the repulsive case $a>0$.  As discussed above, in the general (nonradial) setting there is no ground state for \eqref{NLSa}, and correspondingly there are no heteroclinic orbits at mass-energy $\mathcal{E}_0$. Instead, the scattering/blowup dichotomy persists at the threshold. Indeed, threshold scattering for solutions with constrained kinetic energy has been established in \cite{MiaoMurphyZheng2023}, while the corresponding blowup statement for unconstrained solutions follows from standard virial-type arguments. 

In this work, we study the dynamics for \emph{radial} solutions in the case $a>0$ at the larger \emph{radial} threshold $\mathcal{E}_a$.  In this setting, the result once again parallels the case of $a\leq 0$.  In particular, there exist radial heteroclinic orbits $Q_a^\pm$ that, together with the ground state solution, completely characterize the possible solution dynamics at the threshold.  We state our main result briefly as follows:
\begin{thm}[Radial threshold dynamics, $a>0$]\label{T} The results of Theorem~\ref{T:dynamics1} (i.e. the existence of heteroclinic orbits and classification of dynamics at mass-energy equal to $M(Q_a)E_a(Q_a)$) also hold for $a>0$, provided we restrict to radial solutions. 
\end{thm}

The classification of dynamics for nonlinear dispersive PDE has been an active topic of research for the last several decades.  In particular, there has been substantial interest in the classification of threshold dynamics in recent years.  In fact, the proof of Theorem~\ref{T} relies heavily on the general strategies introduced originally in \cite{DR, DM} and developed further in works such as \cite{yang2026uniqueness, CRITICAL, LiZhang2009, CFR, CP, CM2023, CamposFarahMurphy2026, CamposFarahMurphy2026_Critical}.  The key new ingredient in our setting is a proof of the uniqueness of the radial ground state for \eqref{NLSa}, which underpins the analysis that follows.  Once this fact has been established, we can capitalize on existing techniques quite effectively.  In fact, the overarching radial assumption even simplifies the analysis at several points.

In the rest of the introduction, we will outline the strategy of the proof of Theorem~\ref{T}.  In the main body of the paper, we will focus on presenting in detail the components of the proof that are new in our setting.  For the more standard parts of the argument, we will instead provide proof sketches and refer the reader to the literature for further details. 

\subsection{Outline of the paper}

In Section~\ref{S:notation}, we introduce notation and collect several useful results to be utilized throughout the paper.

We begin our analysis in Section~\ref{S:uniqueness} by establishing the uniqueness of the radial ground state for \eqref{NLSa} with $a>0$.  While previous work has established uniqueness in the case $a\leq 0$ \cite{UNIQ, yang2026uniqueness, kwong1989uniqueness, mcleod1987uniqueness}, the case $a>0$ has so far remained open.  To establish uniqueness, we observe that after a suitable change of variables, we can fit our problem into the general framework introduced in \cite{NK} and obtain uniqueness directly.  

In Section~\ref{S:Spectral}, we carry out a spectral analysis of the operator $\mathcal{L}$ that arises from linearization of \eqref{NLSa} around the ground state solution.  In particular, we identify the eigenfunctions of this operator, which later play a key role in constructing the heteroclinic orbits. We remark that the analysis is simplified here due to the radial assumption. 

In Section~\ref{S:Special}, we construct particular solutions $U^A$ parametrized by $A\in\R$ to \eqref{NLSa} that converge exponentially to the ground state solution forward in time and additionally establish uniqueness properties for such solutions.  In particular, this section constructs the heteroclinic orbits (corresponding to the parameters $A=\pm 1$).  The rest of the paper is primarily dedicated to proving the classification of threshold solutions, which is carried out in Sections~\ref{S:Modulation}--\ref{S:Proof}.

We first consider the case of \emph{constrained} solutions (case (i) in Theorem~\ref{T:dynamics1}). It suffices to show that if a constrained solution fails to scatter in both time directions, then $u=Q_a^-$ up to symmetries.  To this end, we show that if $u$ fails to scatter forward in time, then the orbit of $u$ is pre-compact in $H^1$. Using a virial argument, we can then prove that $u$ converges to the ground state solution exponentially as $t\to\infty$.  By the uniqueness results established in Section~\ref{S:Special}, this ultimately allows us to conclude that $u$ must coincide with the heteroclinic orbit $Q_a^-$.

The case of \emph{unconstrained} solutions (case (iii) in Theorem~\ref{T:dynamics1}) is similar: if $u$ does not blow up forward in time (say), then again we use virial arguments to prove convergence to the ground state solution as $t\to\infty$ and appeal to uniqueness to conclude that $u$ must coincide with $Q_a^+$.

Carrying out these arguments requires several key ingredients.  In Section~\ref{S:Modulation} we introduce the \emph{modulation analysis}, which provides a description of the solution to \eqref{NLSa} when it is close to the orbit of the ground state.  Next, in Section~\ref{S:Constrained}, we consider constrained solutions, proving the compactness property and using virial arguments to obtain exponential convergence to the ground state solution.  In Section~\ref{S:Unconstrained} we then address the case of unconstrained solutions, relying on an improved radial Gagliardo--Nirenberg inequality in place of compactness. 

Finally, in Section~\ref{S:Proof}, we put together all of the pieces and complete the description of the heteroclinic orbits and the classification of threshold dynamics.

\subsection*{Acknowledgements} Much of this work was carried out while L.C. and R.S. were visiting scholars 
at the University of Oregon, with financial support from CNPq (Conselho Nacional de Desenvolvimento Cient\'{i}fico e Tecnol\'{o}gico). L. C. was partially supported by the CNPq grants 07733/2023-8 and 404800/2024-6, and by the FAPEMIG (Funda\c{c}\~{a}o de Amparo \`{a} Pesquisa do Estado de Minas Gerais) grant APQ-03186-24. J. M. was supported by NSF grant DMS-2350225 and Simons Foundation grant MPS-TSM-00006622. R.S. was supported by the CNPq grant 201049/2025-2 and by FAPEMIG.

\section{Preliminaries}\label{S:notation}
\subsection{Notation}
We write $A \lesssim B$ if there exists a constant $C>0$ such that $A \leq C B$. If $A\lesssim B$ and $B\lesssim A$ we write $A\sim B$.  We use the standard mixed Lebesgue space-time norms $L_t^q L_x^r$ on space-time slabs $I\times\R^3$. In later sections we will use the specific norms defined by  
\[
Z(I) := \sup \{L_t^{\frac{10}{3}}H_x^{1,\frac{10}{3}}, L^\infty_tH_x^1, L^2_tH_x^{1,6}, L^5_tH_x^{1,\frac{30}{11}}\}\qtq{and} N(I) := L_t^{\frac{10}{7}} H_x^{1,\frac{10}{7}}.
\]

The equivalence of the Sobolev norms defined in terms of $-\Delta$ and those defined in terms of $\H_a$ (or lack thereof) plays an important technical role throughout this paper.  This question was investigated thoroughly in \cite{SobolevEq}. We record the results we need from \cite{SobolevEq} in the following lemma.

\begin{lemma}[Equivalence of Sobolev spaces \cite{SobolevEq}]\label{L:EquivNorms} Let $a>-\tfrac14$ and $s\in(0,2)$. Define $\sigma=\tfrac12 - ( \tfrac14+a )^{\frac12}$. If $p\in(1,\infty)$ satisfies $\tfrac{s+\sigma}{3}<\tfrac{1}{p}<\min\{1,1-\frac{\sigma}{3}\}$, then
\[
\| |\nabla|^s f\|_{L^p} \lesssim_{p,s} \| (\H_a)^{\frac{s}{2}} f\|_{L^p} \qtq{for all}f\in C_c^\infty(\R^3\backslash\{0\}).
\]
If $p\in(1,\infty)$ satisfies $\max\{\tfrac{s}{3},\tfrac{\sigma}{3}\}<\tfrac{1}{p}<\min\{1,1-\tfrac{\sigma}{3}\}$, then
\[
\| (\H_a)^{\frac{s}{2}} f\|_{L^p} \lesssim_{p,s} \| |\nabla|^s f\|_{L^p} \qtq{for all} f\in C_c^\infty(\R^3\backslash\{0\}). 
\]
\end{lemma}

We next record an improved Gagliardo--Nirenberg inequality for radial functions. 

\begin{lemma}[Radial Gagliardo--Nirenberg inequality \cite{KVMZ}]\label{L:RadialGNI}
For any $f\in H_{\text{rad}}^1(\R^3)$ and $R>1$,
\[
\|f\|_{L_x^4(|x|>R)}^4\lesssim R^{-2}\|f\|_{L_x^2}\|f\|_{\dot{H}_a^1}^3.
\]
\end{lemma}

We will need the following profile decomposition adapted to the $H^1\to L_{t,x}^5$ Strichartz estimate.  Because we work in the radial setting, there is no failure of compactness due to spatial translation.

\begin{prop}[Linear Profile Decomposition, \cite{KVMZ}]\label{P:LPD} Let $a>-\frac{1}{4}$ and let $\{f_n\}$ be a bounded sequence in $H^1_{\text{rad}}$. Passing to a subsequence, there exist $J^*\in\{1,2,...,\infty\}$, nonzero radial profiles $\{\phi^j\}_{j=1}^{J^*}\subset H^1(\R^3)$, and parameters $(t_n^j)_{j=1}^{J^*}\subset \R$ satisfying the following:

For each finite $0\leq J\leq J^*$, we can write
\[
f_n=\sum_{j=1}^J\phi_n^j+r_n^J,\quad\text{with}\quad \phi_n^j=e^{it_n^j\H_a}\phi^j\quad\text{and}\quad r_n^J\in H^1
\]
satisfying the following decouplings for any finite $0\leq J\leq J^*$:
\begin{align*}
&\lim_{n\to\infty}\biggl\{\|(\H_a)^\frac{s}{2}f_n\|_{L_x^2}^2-\sum_{j=1}^J\|(\H_a)^\frac{s}{2}\phi_n^j\|_{L_x^2}^2-\|(\H_a)^\frac{s}{2}r_n^J\|_{L_x^2}^2\biggr\}=0,\quad s\in\{0,1\}, \\
&\lim_{n\to\infty}\biggl\{\|f_n\|_{L_x^4}^4-\sum_{j=1}^J\|\phi_n^j\|_{L_x^4}^4-\|r_n^J\|_{L_x^4}^4\biggr\}=0.
\end{align*}

The remainders $r_n^J$  vanish in the Strichartz norm:
\begin{equation}\label{rnjvanish}
\limsup_{J\to J^*}\limsup_{n\to\infty}\|e^{-it\H_a}r_n^J\|_{L_{t,x}^5(\R\times\R^3)}=0,
\end{equation}
and for any $j\not=k$ we have the asymptotic orthogonality
\begin{equation}\label{tnjorthogonal}
\lim_{n\to\infty}|t_n^j-t_n^k|=\infty.
\end{equation}
Finally, for each $j$, we may assume that either $t_n^j\to\pm\infty$ or $t_n^j\equiv0$.
\end{prop}

We next import a stability result for \eqref{NLSa} from \cite[Section~2]{KVMZ}. We write $\dot S_a^{\sigma}$ to denote homogeneous Strichartz norms adapted to $\H_a$ and set 
\[
\tilde N(I):=L_{t,x}^\frac{10}{7}(I\times\R^3)+L_t^\frac{5}{3}L_x^\frac{30}{23}(I\times\R^3)+L_t^1L_x^2(I\times\R^3).
\]

\begin{prop}[Stability \cite{KVMZ}]\label{P:Stab}
Fix $a>-\frac{1}{4}$. Let $\tilde{v}:I\times\R^3\to\mathbb{C}$ solve
\[
(i\partial_t-\H_a)\tilde{v}=-|\tilde v|^2\tilde v+e,\quad v(t_0)=\tilde v_0\in H_x^1(\R^3)
\]
for some $t_0\in I$ and some $e:I\times\R^3\to \mathbb{C}$. Fix $v_0\in H^1$ and suppose that 
\[
\|v_0\|_{H^1}+\|\tilde v_0\|_{H^1}\leq E\quad\text{and}\quad \|\tilde v\|_{L_{t,x}^5}\leq L
\]
for some $E,L>0$. There exists $\varepsilon_0=\varepsilon_0(E,L)>0$ such that if $0<\varepsilon<\varepsilon_0$ and
\[
\|\tilde{v}_0-v_0\|_{\dot{H}^\frac{1}{2}}+\||\nabla|^\frac{1}{2}e\|_{\tilde N(I)}<\varepsilon,
\]
then there exists a solution $v:I\times\R^3\to\mathbb{C}$ to \eqref{NLSa} with $v(t_0)=v_0$ satisfying
\[
\|v-\tilde{v}\|_{\dot{S}_a^\frac{1}{2}}\lesssim_{E,L}\varepsilon\qtq{and}
\|v\|_{S_a^1(I\times\R^3)}\lesssim_{E,L}1.
\]
\end{prop}

\section{Uniqueness for the radial ground state}\label{S:uniqueness}

In the repulsive case $a>0$, the Gagliardo--Nirenberg inequality does not admit optimizers in $H^1$ (cf. \cite[Theorem 3.1]{KVMZ}).  If one restricts to radial functions, however, then compactness is restored and one obtains a radial optimizer denoted $Q_a$ (which then yields the ground state solution $e^{it}Q_a$ to \eqref{NLSa}):
 
\begin{thm}[Sharp Gagliardo-Nirenberg inequality, \cite{KVMZ}]\label{T:GN} Let $a>0$ and set
 \[
 C_a := \sup \left\{ \|f\|_{L^4}^4 \div \bigl[\|f\|_{L^2} \|f\|_{\dot{H}^1_a}^3\bigr] : f \in H_{\text{rad}}^1\backslash\{0\}\right\}.
 \]
 Then $C_a \in (0, \infty)$ and equality is attained by a function $Q_a\in H_a^1$, which is a nonnegative radial solution to 
 \begin{equation} \label{GroundStateEq}
-\H_a Q_a - Q_a + Q_a^3 = 0.
 \end{equation}
Additionally, we have the Pohozaev identities 
\begin{equation} \label{pohozaev}
 \|Q_a\|_{L^2}^2 = \tfrac{1}{3} \| Q_a\|_{\dot{H}^1_a}^2 = \tfrac{1}{4} \|Q_a\|_{L^4}^4 . 
\end{equation}
\end{thm}

The goal of this section is to establish the uniqueness of the ground state $Q_a$.

\begin{thm}[Uniqueness of the radial ground state]\label{T:UniqGS}
Let $a>0$.  There exists a unique nonnegative radial $H^1$ solution to \eqref{GroundStateEq}.
\end{thm}

The existence of at least one nonnegative, radial $H^1$ solution is guaranteed by Theorem~\ref{T:GN}.  Hence, it remains only to establish uniqueness. We begin by letting $Q$ denote any such solution and use the equation to derive some asymptotic properties.  

Writing $Q=Q(r)$ as a function of the radial variable $r\in(0,\infty)$, we first rewrite \eqref{GroundStateEq} as 
\begin{equation}\label{GS_positive_a}
Q_{rr} + \tfrac{2}{r} Q_r -\tfrac{a}{r^2} Q +Q^3 -Q =0,
\end{equation}
where subscripts denote differentiation with respect to $r$.

The arguments of \cite{UNIQ} already show that $Q$ decays faster than $r^{-m}$ for any $m>0$ as $r\to\infty$. In fact, we can prove that $Q$ decays exponentially at infinity: 
\begin{prop}\label{Expdecay} For any $\epsilon>0$ and $r\geq 1$,  
\[
Q(r) \lesssim e^{-(1-\epsilon)r}.
\]
\end{prop}

\begin{proof} Define $v(r)=rQ(r)$, which satisfies $v(r)\to 0$ as $r\to\infty$ and 
\[
v_{rr} - W(r) v = 0,\qtq{where} W(r):=1+\tfrac{a}{r^2} -\tfrac{v^2}{r^2}.
\]
Now, given $\epsilon>0$, let $R_0\geq 1$ be large enough that $W(r)\geq 1-\epsilon$ for any $r\geq R_0$.  Let us also choose $C_0$ large enough that
\[
v(R_0)\leq C_0 e^{-\lambda R_0},\qtq{where} \lambda:=\sqrt{1-\epsilon}. 
\]

We now introduce
\[
z(r) = v(r) - C_0 e^{-\lambda r},
\]
which satisfies $z(R_0)\leq0$ and $z(r)\to 0$ as $r\to\infty$. Thus
\[
z''(r)-W(r)z(r) = -C_0 e^{-\lambda r}(\lambda^2-W(r))>0 \qtq{for all} r>R_0. 
\]
We now claim that $z(r)\leq 0$ for all $r\geq R_0$. Indeed, suppose instead that $z$ attains a positive maximum at $r=r_0\geq R_0$, so that $z''(r_0)\leq 0$.  As $W(r_0)>0$, this yields
\[
z''(r_0)-W(r_0)z(r_0)<0,
\]
a contradiction.  We therefore obtain $z(r)\leq 0$ for $r\geq R_0$, which in turn implies
\[
Q(r) \lesssim \tfrac{1}{r}e^{-\lambda r},
\]
as desired. \end{proof}

We next consider the behavior of $Q$ at $r=0$. We adopt the strategy in \cite{yang2026uniqueness}, which addressed the case $a<0$. Following the arguments of \cite[Lemma~2.5]{yang2026uniqueness} directly, we can obtain the following.

\begin{lem}[Bounds at $r=0$, \cite{yang2026uniqueness}]\label{L:FirstAsymp} Define
\[
u_1(r) = r^{\frac12}Q(r) \qtq{and} u_2(r) = r^{\frac32} Q'(r).
\]
Then 
\[
\lim_{r\to 0} u_j(r) = 0 \qtq{for}j=1,2.
\]
\end{lem}

We can then obtain the following asymptotic behavior as $r\to 0$.

\begin{cor} Let 
\[
\beta=\sqrt{1+4a}.
\]
Then
\[
Q(r) \sim r^{-\frac12+\frac12\beta} \qtq{and} Q'(r) \sim r^{-\frac32+\frac12\beta}\qtq{as}r\to 0. 
\]
\end{cor} 

\begin{proof} Let $u_1,u_2$ be as in Lemma~\ref{L:FirstAsymp}.  We change variables via $s=-\ln(r)$ and define $v_1,v_2$ by
\[
v_1 = \tfrac{1}{\beta}\bigl[(-\tfrac{1}{2}+\tfrac{\beta}{2})u_1-u_2\bigr]\qtq{and} v_2 = \tfrac{1}{\beta}\bigl[(\tfrac{1}{2}+\tfrac{\beta}{2})u_1+u_2\bigr].
\]
Letting $'$ denote $\tfrac{d}{ds}$, we then obtain the following system for $v_1,v_2$:
\begin{align*}
v_1' &= \tfrac{\beta}{2}v_1 + \tfrac{1}{\beta} e^{-2s}(v_1+v_2) - \tfrac{1}{\beta} e^{-s}|v_1+v_2|^2(v_1+v_2), \\
v_2' &=-\tfrac{\beta}{2}v_2 - \tfrac{1}{\beta} e^{-2s}(v_1+v_2)+\tfrac{1}{\beta} e^{-s}|v_1+v_2|^2(v_1+v_2).
\end{align*}
Using Lemma~\ref{L:FirstAsymp}, we have that $v_j\to 0$ as $s\to\infty$ for $j=1,2$. Thus, multiplying the first equation by $e^{-\frac12\beta s}$ and integrating, we can obtain that 
\[
v_1(s)=\mathcal{O}(e^{-s})\qtq{as} s\to\infty.
\]
Similarly, multiplying the second equation by $e^{\frac12\beta s}$ and integrating to infinity, we can derive that 
\[
v_2(s) = C e^{-\frac12\beta s}+\mathcal{O}(e^{-s})\qtq{as}s\to\infty. 
\]

With these bounds in hand, we can iterate to obtain
\[
v_1(s) = \mathcal{O}_k(e^{-ks}) \qtq{and} v_2(s) = Ce^{-\frac12\beta s}+ \mathcal{O}_k(e^{-ks})
\]
for some $k\geq \tfrac12\beta$ and some $C\neq 0$.  Returning to the variables $u_1$ and $u_2$, we obtain $u_j(r(s)) \sim e^{-\frac12\beta s}$, which translates back into the desired asymptotic behavior for $Q$ and $Q'$ as $r\to 0$. \end{proof}

We now turn to the question of uniqueness.  We will show that in the radial regime with $a>0$, uniqueness may be derived from the following theorem, due to Shioji and Watanabe (see \cite[Theorem~1]{NK}). 

\begin{thm}[Uniqueness \cite{NK}]\label{ThmUniqueness} Consider the problem 
\begin{equation}\label{Problem1}
\begin{cases}
v_{rr} + \frac{f_r(r)}{f(r)} v_r + g(r)v + h(r) v^3=0, \quad 0 < r < \infty,\\
v(0) \in (0, \infty), \quad \lim_{r\to\infty} v(r) = 0,
\end{cases}
\end{equation}
where $f,g,h:(0,\infty)\to\R$. Define 
\begin{align*}
\begin{cases}
a(r) = f(r)^{\frac{4}{3}}h(r)^{-\frac{1}{3}}\\
b(r) = -\frac{1}{2}a_r(r) + \frac{f_r(r)}{f(r)}a(r)\\
c(r) = -b_r(r) +\frac{f_r(r)}{f(r)}b(r) \\
G(r) = -b(r)g(r) + \frac{1}{2}c_r(r) + \frac{1}{2}(ag)_r(r).
\end{cases}
\end{align*}
Assume the following conditions hold:
\begin{itemize}
\item[$(1)$] $\displaystyle\lim_{r \to 0^+} f(r) < \infty$,
\vspace{2pt}
\item[$(2)$] $\displaystyle\lim_{r \to 0^+} \frac{1}{f(r)}\displaystyle\int_0^r f(\tau) (|g(\tau)| + h(\tau))\, d\tau = 0$,
\vspace{3pt}
\item[$(3)$] There exists $R \in (0,\infty)$ such that 
\begin{itemize}
\vspace{2pt}
\item[$(i)$] $fg, fh \in L^1(0,R)$,
\vspace{2pt}
\item[$(ii)$] $\tau \mapsto f(\tau)(|g(\tau)| + h(\tau)) \displaystyle\int_\tau ^R \frac{d \rho}{f(\rho)} \in L^1(0,R)$, and 
\item[$(iii)$] $\displaystyle\frac{1}{f} \notin L^1(0,R)$,
\end{itemize}
\vspace{2pt}
\item[$(4)$] $a(r), |b(r)|, a(r)g(r)$ and $a(r)h(r) \to 0$ when $r \to 0$,
\vspace{3pt}
\item[$(5)$] $\displaystyle\lim_{r \to 0}c(r) \in [0,\infty]$,
\vspace{2pt}
\item[$(6)$] There exists $R \in [0,\infty)$ such that $G(r) \geq 0$ on $(0,R)$, and $G(r) \leq 0$ on $(R, \infty)$,
\vspace{2pt}
\item[$(7)$] $G^-(r) := \min\{G(r),0\} \not\equiv 0$.
\end{itemize}
Then \eqref{Problem1} has at most one positive solution.
\end{thm}

\begin{proof}[Proof of Theorem~\ref{T:UniqGS}] We continue to denote $\beta=\sqrt{1+4a}$. Under the change of variables $v(r) = r^{\frac12-\frac{\beta}{2}} Q(r)$, the equation \eqref{GS_positive_a} takes the form
\[
v_{rr} + \tfrac{f_r(r)}{f(r)} v_r + g(r)v + h(r) v^3=0
\]
where
\[
f(r)=r^{1+\beta},\quad g(r) = -1,\qtq{and} h(r) = r^{-1 +\beta}.
\]
To obtain uniqueness, we now check conditions (1)--(7). 

Condition (1) is immediate, by inspection.

For condition $(2)$, we first compute
\begin{align}\label{Cond1}
f(\tau)(|g(\tau)| + h(\tau)) = \tau^{1+\beta} + \tau^{2\beta},
\end{align}
so that
\[
\int_0^r f(\tau)(|g(\tau)| + h(\tau)) d\tau = \tfrac{r^{2+\beta}}{2+\beta} + \tfrac{r^{1+2\beta}}{1+2\beta}.
\]
Dividing by $f(r)$ and letting $r \to 0^+$ yields $(2)$.

To verify condition $(3)$, we let $R=1$ and first note that $(i)$ and $(iii)$ follow immediately from the fact that $a>0$.  To check $(ii)$, we first compute
\[
\int_\tau ^1\tfrac{d \rho}{f(\rho)} = \tfrac{1}{\beta}\left( \tau^{-\beta}-1\right).
\]
Combining this with \eqref{Cond1}, we have
\begin{align*}
f(\tau)(|g(\tau)| + h(\tau)) \int_\tau ^1 \tfrac{d \rho}{f(\rho)} &= \tfrac{1}{\beta}(\tau +\tau^{\beta} - \tau^{1+\beta}-\tau^{2\beta})\in L^1(0,1).
\end{align*}

For condition $(4)$, we compute 
\[
a(r) = r^{\frac{5}{3} +\beta} \qtq{and} b(r) = (\tfrac{1}{6}-\tfrac{1}{2}\beta)r^{\frac{2}{3}+\beta},
\]
from which we derive that $a,b,ag,ah\to 0$ as $r\to 0$. 

For condition $(5)$, we compute
\[
c(r) =  -b_r(r) +\tfrac{f_r(r)}{f(r)}b(r) = \tfrac{1-3\beta}{18} r^{-\frac{1}{3} +\beta} \to 0 \qtq{as}r\to0.
\]

Finally, for $(6)$ and $(7)$, we compute 
\[
G(r) = -\tfrac{r^{\beta - \frac{4}{3}}}{108} \left[ 36\left(3\beta+2\right) r^2 + \left(3\beta-1\right)^2 \right].
\]
As $a>0$, we obtain condition $(6)$ with $R=0$, along with condition $(7)$.

Having verified conditions $(1)$--$(7)$, we may now apply Theorem~\ref{ThmUniqueness} to conclude uniqueness.\end{proof}

To conclude this section, we record some useful properties of the ground state and consequences of the sharp Gagliardo--Nirenberg inequality.

\begin{prop} \label{boundsGS} Let $a>0$ and $\beta=\sqrt{1+4a}$.  Then
\[
Q_a \in L_x^r \qtq{for} r\in[1,\infty] \qtq{and} \nabla Q_a \in L_x^r \qtq{for} r\geq 1\qtq{s.t.} (3-\beta)r<6. 
\]
\end{prop}

In the following proposition, we recall the notation $\mathcal{E}_a$ and $\mathcal{K}_a$ from Theorem~\ref{T:KMVZ}.

\begin{prop}\label{P:Coercivity}
Let $a>0$ and suppose $u_0\in H_{\text{rad}}^1$ satisfies $M(u_0)E_a(u_0)\leq \mathcal{E}_a.$  Let $u:I\times\R^3\to \mathbb{C}$ be the maximal-lifespan solution to \eqref{NLSa}.
\begin{itemize}
\item[$(i)$] If $\|u_0\|_{L^2}\|u_0\|_{\dot{H}_a^1}<\mathcal{K}_a$, then $I=\R$ and $\|u(t)\|_{L^2}\|u(t)\|_{\dot{H}_a^1}<\mathcal{K}_a$ for all $t\in\R$. 
\item[$(ii)$] If $\|u_0\|_{L^2}\|u_0\|_{\dot{H}_a^1}=\mathcal{K}_a$, then $u(t)=e^{it}Q_a$ up to symmetries. 
\item[$(iii)$]  If $\|u_0\|_{L^2}\|u_0\|_{\dot{H}_a^1}>\mathcal{K}_a$, then $\|u(t)\|_{L^2}\|u(t)\|_{\dot{H}_a^1}>\mathcal{K}_a$ for all $t\in I$. 
\end{itemize}
\end{prop}

We will refer to solutions as in $(i)$ \emph{constrained} and to solutions as in $(iii)$ \emph{unconstrained}.

\section{Spectral Analysis}\label{S:Spectral}

This section contains an analysis of the operator arising from linearization of \eqref{NLSa} around the ground state solution. Given $a >0$ and a radial solution $u$ to \eqref{NLSa}, we define $v=v_1+iv_2$ by setting $u = e^{it} (Q_a + v)$ and obtain the following equation:
\begin{align}\label{lin1}
i \partial_t v -\H_av  - v  + K(v) + R(v) = 0
\end{align}
where 
\begin{equation}\label{K(v)R(v)}
K(v) := Q_a^2(3v_1 + i v_2)\qtq{and} R(v) := Q_a^3 G(Q_a^{-1} v),
\end{equation}
with 
\[
G(z) := 3z_1^2 + z_2^2 + z_1^3 + z_1 z_2^2 + i(2z_1 z_2 + z_1^2 z_2 + z_2^3).
\]

In order to express \eqref{lin1} as a system for $[v_1\ v_2]^t$, we first introduce the bounded symmetric operator $L: H^1_{\text{rad}} \times H^1_{\text{rad}} \to H^{-1}_{\text{rad}} \times H^{-1}_{\text{rad}}$ given by 
\begin{equation}\label{D:L}
L :=  \begin{pmatrix*}[c]
L_+ & 0 \\
0 & \,\, L_-
\end{pmatrix*},
\end{equation}
where 
\[
L_+ := 1 + \H_a - 3Q_a^2\qtq{and} L_- := 1 + \H_a - Q_a^2.
\]
We further introduce the densely-defined operator $J : D(J) \subset H^{-1}_{\text{rad}} \times H^{-1}_{\text{rad}} \to H^1_{\text{rad}} \times H^1_{\text{rad}}$ given by
\[
J = \begin{pmatrix*}[r]
0 & 1\, \\
-1 & \,\,\,0\,
\end{pmatrix*}.
\]
Using the natural identification of $a + ib \in \mathbb{C}$ with $[a, b]^t \in \mathbb{R}^2$, the linearized equation \eqref{lin1} can therefore be written 
\begin{align}\label{linearized}
\partial_t v + \mathcal{L} v = i R(v),\qtq{where} \mathcal{L}=JL.
\end{align}

Our interest is in the spectral properties of $\mathcal{L}$. We note immediately that $L_- Q_a = 0$ and $L_+ Q_a = -2 Q_a^3$, so that $(L_+Q_a,Q_a) < 0$ (where here and below $(\cdot,\cdot)$ denotes the $L^2$ inner product). 

We first show that there exist only two nonpositive directions for $L$. 
\begin{prop} \label{coercivity}
Let $v \in H^1_{\text{rad}}(\mathbb{R}^3; \mathbb{R})$. Then
\begin{enumerate}
\item[(i)] $(L_- v, v) \gtrsim \|v\|_{H^1}^2$ provided $(v, Q_a) = 0$, 
\item[(ii)] $(L_+ v, v) \gtrsim \|v\|_{H^1}^2$ provided $(v, \H_a Q_a) = 0$.
\end{enumerate}
\end{prop}
\begin{proof} The proof is an adaptation of the proof of \cite[Proposition~3.1]{CRITICAL}.  We will equivalently prove $L^2$-coercivity under the appropriate orthogonality conditions for the operators $\grada^{-1} L_{\pm} \grada^{-1} $. Writing $\grada^{-1} L_{\pm} \grada^{-1} = I - \mathcal{K}_\pm $, we first prove that $\mathcal{K}_\pm$ are compact on $L^2$. Since $\grada^{-1} $ is continuous, it suffices to prove that $Q_a^2 \grada^{-1}$ is compact.  To this end, we first note that this operator is bounded.  Indeed, by Sobolev embedding,
\begin{align*}
\| Q^2_a \grada^{-1}v\|_{L^2} &\lesssim \|Q_a^2\|_{L^3} \|\grada^{-1}v\|_{L^6} \lesssim \|v\|_{L^2}.
\end{align*}
We next observe that
\begin{align*}
\| |\nabla|^{\frac12} &(Q_a^2 \grada^{-1}v) \|_{L^2} \\
&\lesssim \| \nabla(Q_a^2 \grada^{-1}v) \|_{L^{\frac{3}{2}}} \\
&\lesssim \|Q_a\|_{L^{12}}^2 \|\nabla \grada^{-1} v\|_{L^2} + \|Q_a \nabla Q_a\|_{L^2}\| \grada^{-1} v \|_{L^6} \lesssim \|v\|_{L^2}.
\end{align*}
Finally, using the decay of the ground state away from the origin, we have 
\begin{align*}
\| Q_a^2 \grada^{-1} v \|_{L^2(|x|>R)} \lesssim R^{-1} \|v\|_{L^2}. 
\end{align*}
The desired compactness now follows from Rellich--Kondrachov. 

It follows that the eigenvalues of $I-\mathcal{K}_\pm$ are discrete, and since the identity acts as a spectral shift, they can only accumulate at 1. Moreover, $I-\mathcal{K}_\pm$ has a finite number of eigenvalues $\{\lambda_i^\pm\}_{i=1}^N$ in $(-\infty, \tfrac12]$, which we assume are nondecreasing. Finally, by Weyl's Theorem, the essential spectrum is $[1, \infty)$. 

We first prove that there are orthogonal conditions that guarantee that $L_\pm$ are nonnegative.  By construction, $Q_a$ minimizes the Weinstein functional 
\[
W(f) = \frac{\left(\displaystyle \int | \nabla f|^2\,dx + \int \frac{a |f|^2}{|x|^2 } \,dx \right)^{3/2} \left( \displaystyle \int |f|^2 \,dx\right)^{1/2}}{ \displaystyle\int |f|^4\,dx}.
\]
Thus $\frac{d^2}{d \ep^2} W(Q_a + \ep v)|_{\ep=0} \geq 0$. Writing $v = v_1 + iv_2$ and assuming $(v_1,\H_a Q_a)=0$, we apply \eqref{pohozaev} to obtain
\begin{align*}
(L_+v_1, v_1) + (L_- v_2, v_2) \geq \frac{1}{\| Q_a\|_2^2}  \left( \int Q_a v_1\right)^2.
\end{align*}
It follows that $L_-$ is always nonnegative and that $L_+$ is nonnegative if $(v_1, \H_a Q_a) = 0$.  Now, by \eqref{GroundStateEq}, we have that $\lambda_1^- = 0$, and we have now shown that although $\lambda_1^+ <0$, we have $\la_2^+ \geq 0$. To finish the proof, it will therefore suffice to show that $\la_2^\pm >0$. To this end, we will prove that 
\[
\text{ker} L_+=\emptyset\qtq{and}\text{ker} L_-=\text{span}\{Q_a\}.
\]

Let us first prove that $\text{ker} L_- = \text{span}\{Q_a\}$.\footnote{We remark that in contrast to the case $a<0$ addressed in \cite{yang2026uniqueness}, we are working in the radial setting and hence do not need to consider the spherical harmonic expansion.}  Suppose $f_0$ were another independent radial solution to $L_- f = 0$. Thus, 
\begin{align*}
(f_0)_{rr} + \tfrac{2}{r}(f_0)_r - \frac{a}{r^2} f_0 + Q_a^2 f_0 -f_0 = 0.
\end{align*}
Taking the Wronskian and using Abel's theorem, we find
\begin{align} \label{abel}
(f_0)_r Q_a - (Q_a)_r f_0 = \tfrac{C}{r^2},\qtq{so that}\bigl(\tfrac{f_0}{Q_a}\bigr)_r = \tfrac{C}{Q_a^2 r^2}.
\end{align}
From the asymptotics for ground state, we know that $Q_a \sim r^{-\frac{1}{2} + \frac{1}{2}\beta}$ as $r\to 0$, where we denote $\beta=\sqrt{1+4a}$ as above. Thus, integrating the above equality and using \eqref{abel}, we obtain 
\[
f_0 \sim r^{-(\frac{1}{2} + \frac{1}{2}\beta)}\qtq{and} (f_0)_r \sim r^{-(\frac{3}{2} + \frac{1}{2}\beta)} \qtq{as}r\to 0.
\]
In particular, $f_0\notin H^1$.  It follows that $\text{ker}(L_{-}) = \text{span}(Q_a)$. 

We next prove that $\text{ker} L_+=\emptyset$. We suppose that $L_+$ has a nonzero radial element $v_0$ in its kernel. As $L_+$ has a single negative eigenvalue, Sturm--Liouville theory guarantees that $v_0$ has a single zero at some $r=r_0$. Now observe that $Q_a$ and $Q_a^3$ are in the range of $L_+$; indeed, we have
\[
L_+Q_a = -2Q_a^3 \qtq{and} L_+(Q_a+x\cdot\nabla Q_a)=-2Q_a.
\]
Thus, since $L_+$ is self-adjoint, we have $v_0\perp Q_a,Q_a^3$.  In particular,
\[
 \int v_0(r)[Q_a^3(r) - Q_a^2(r_0)Q_a(r)]\,dr=0.
\]
However, because $Q_a$ is strictly decreasing and $v_0$ changes sign only at $r_0$, the integrand above has a constant sign. Thus we obtain a contradiction and conclude that $L_+$ has no radial kernel.

Finally, let us upgrade to the coercivity statement for $L_\pm$. Writing $\tilde L_\pm = \grada^{-1} L_\pm \grada^{-1} = I+\mathcal{K}_\pm$, we suppose towards a contradiction that there exists an $L^2$-normalized sequence $v_n$ such that $\langle \tilde L_\pm v_n,v_n\rangle \to 0$, with $v_n$ satisfying the appropriate orthogonality condition depending on the $\pm$ sign. We may assume that there exists $v$ such that $v_n\rightharpoonup v$ weakly in $L^2$, and note that $v$ then satisfies the same orthogonality condition as the $v_n$.  Recalling that $\mathcal{K}_\pm$ are compact on $L^2$, we now obtain that $\mathcal{K}_\pm v_n \to \mathcal{K}_\pm v$ strongly in $L^2$. Using weak lower-semicontinuity of the norm, we therefore obtain
\[
\langle \tilde L_\pm v,v\rangle \leq \liminf_{n\to\infty} \langle \tilde  L_\pm v_n,v_n\rangle =0.
\]
By the nonnegativity established above, this forces $L_\pm v =0$ as well as $\|v\|_{L^2}^2=1$.  In particular, we have the strong convergence $v_n\to v$ in $L^2$ and $v\neq 0$.  But then in the case of $L_-$, we obtain a contradiction due to the fact that $v\perp Q_a$, while in the case of $L_+$ we obtain a contradiction due to the fact that $\text{ker}L_+=\emptyset$. \end{proof}

Recalling the operators $L$ and $\mathcal{L}$ defined in \eqref{D:L} and \eqref{linearized}, we now turn to a description of the spectrum of the matrix operator $\mathcal{L}$. 

\begin{prop} \label{spectral}
The operator $\mathcal{L}$ defined in \eqref{linearized} generates a $C^0$ group $e^{t\mathcal{L}}$ of bounded linear operators in $H_{\text{rad}}^1(\R^3)\times H_{\text{rad}}^1(\R^3)$ and there is a decomposition 
\begin{align*}
H^1_{\text{rad}}(\R^3) = E^u \oplus E^s \oplus E^c
\end{align*}
into the unstable, center and stable spaces of $\mathcal{L}$ satisfying:
\begin{enumerate}
\item $E^u, E^c, E^s \subset dom(\mathcal{L})$ are invariant under $e^{t \mathcal{L}}$.
\item $dim \, E^u = dim \, E^s = 1$.
\item There exists $e_0>0$ such that 
\begin{align}
&e^{t \mathcal{L}}v = e^{e_0t} v \quad \forall v\in E^u, \\
&e^{t \mathcal{L}}w = e^{-e_0t} w \quad \forall w\in E^c.
\end{align}
\end{enumerate}
\end{prop}
\begin{proof} Again, we follow the lead of \cite{CRITICAL}. The decomposition of $H^1_{\text{rad}}(\R^3)$ and item $(1)$ follow directly from \cite[Theorem~2.2]{SPECTRAL}. 

Recalling from the proof of the previous proposition that $L$ has exactly one negative direction, \cite[Theorem~2.3]{SPECTRAL} states that 
\[
\dim \, E^u + k^{\leq 0}_0 = 1,
\]
where $k^{\leq 0}_0$ is the number of nonpositive directions of $\langle L \cdot, \cdot\rangle$ on the generalized kernel of $\mathcal{L}$ modulo $\text{ker} L$. In order to prove that $k^{\leq 0}_0 = 0$, we suppose $v$ satisfies $(JL)^2 v =0$ and $Lv\neq 0$ and we must verify that $\langle Lv,v\rangle>0$. Writing $v=[f\ g]^T$, we therefore assume $L_-L_+ f=0$ and $L_+ L_-g=0$. By the analysis above, we must have $L_+f = cQ_a$ and $L_-g=0$, and we may assume $g=0$.  On the other hand, $L_+f = cQ_a$ guarantees that $c\neq 0$, and in fact we must have $f=-\tfrac{c}{2}(Q_a+x\cdot\nabla Q_a)$ (as this is the unique solution to $L_+f = cQ_a$).  The desired positivity then follows from directly verifying that 
\[
(L_+(Q_a+x\cdot\nabla Q_a),Q_a+x\cdot\nabla Q_a)>0.
\]
Thus $\dim E^u =1$, which yields (2). 

Finally, (3) follows from \cite[Theorem~2.1(6)]{SPECTRAL}.  It follows that there exists $e_0 \in \mathbb{C}$ with $\Re(e_0)>0$ and $Y_+ = Y_1 + i Y_2$ in $H^1_{\text{rad}}(\R^3)$ satisfying $L_- Y_2 = e_0 Y_1$ and $L_+Y_1 = -e_0 Y_2$. As $L_{\pm}$ are real valued operators, we conclude that $e_0 \in \R$. \end{proof}

\begin{remark}\label{H1dec} By the preceding proposition, we may define $Y_- = \overline{Y_+}$ so that $\mathcal{L}Y_\pm = \pm e_0 Y_\pm$. The results of \cite{SPECTRAL} then imply the following orthogonal decomposition:
\[
 H^1_{\text{rad}}(\mathbb{R}^3) = \bigoplus_{j=0}^3 X_j
\]
where $X_0 = \text{ker}\,L$, $X_1 = \text{span}\{Y_+\}$, $X_2 = \text{span}\{Y_-\}$, and
\[
X_3 = \{ u \in  (H_{\text{rad}}^1 \times H_{\text{rad}}^1)\backslash \text{ker}\,L: \langle L u, v\rangle = 0 \qtq{for all} v \in X_1 \oplus X_2 \}.
\]
Moreover, we have coercivity $\langle L|_{{X_3}}v, v\rangle \gtrsim \|v\|_{H^1}^2$. We note that the authors of \cite{CRITICAL} also utilized the results of \cite{SPECTRAL} in an analogous way.
\end{remark}

To conclude this section we derive some decay properties related to the eigenfunctions $Y_\pm\in H_a^2$ given above, in the spirit of those obtained in \cite{DM} for the standard NLS.   We note that for $a>0$, Lemma~\ref{L:EquivNorms} yields $Y_\pm\in H_a^{\frac32+}\subset H^{\frac32+}\subset L^\infty$.  In fact, we can obtain arbitrarily good decay away from zero.

\begin{prop}\label{Asymptoticseigen}
If $\phi \in C^\infty_c(\mathbb{R}^3 \backslash \{0\})$ then
\begin{equation}\label{Y-localized-away}
\| \phi(\tfrac{x}{R}) Y_\pm \|_{H^k} \lesssim_{\phi,k,\ell} R^{-\ell} \qtq{for any}\ell\geq 0.
\end{equation}
\end{prop}
\begin{proof} 
As this part is fairly standard, we omit most of the details.  The basic idea is to write the system satisfied by $Y_1$ and $Y_2$ and utilize an elliptic regularity argument to demonstrate arbitrarily good decay away from zero. Indeed, we have 
\begin{align*}
\begin{cases}
(1- \Delta)Y_1 = -\frac{a}{|x|^2} Y_1 + 3 Q_a^2Y_1 - e_0 Y_2, \\
(1- \Delta)Y_2 = -\frac{a}{|x|^2} Y_2 +  Q_a^2Y_2 + e_0 Y_1.
\end{cases}
\end{align*}
We can therefore apply an inductive argument to obtain arbitrarily fast decay of the $H^k$ norms of $Y_1$ and $Y_2$.
\end{proof}

One additionally needs the following decay estimates in the construction of the heteroclinic orbits below. 

\begin{prop}\label{F_reg}
If $\lambda \in \mathbb{R}\backslash \sigma(\mathcal{L})$ and $F \in L^2(\mathbb{R}^3)$ satisfies estimates of the form \eqref{Y-localized-away}, then the solution $f \in H^2_a\subset H^1$ to $(\mathcal{L}-\lambda) f = F$ satisfies the same estimate. Moreover, $f \in H^{1,\frac{30}{11}}$.    
\end{prop}
\begin{proof} The first assertion can be established via the argument used to obtain Proposition~\ref{Asymptoticseigen}.  Let us therefore focus on proving $f\in H^{1,\frac{30}{11}}$, and in fact let us simply show the result for the component $f_1$, which satisfies
\[
\H_a f_1 = -f_1 + 3Q_a^2 f_1 + \lambda f_2 + F_2.
\]
Using Lemma~\ref{L:EquivNorms}, self-adjointness of $\H_a$, and the equation above, we obtain  
\begin{align*}
\| |\nabla|^{\frac75} f_1 \|_2^2 & \lesssim ( \H_a^{\frac{7}{10}} f_1,\H_a^{\frac{7}{10}}f_1) \\
& = ( \H_a f_1,\H_a^{\frac 25} f_1) \\
& \lesssim |( f_1,\H_a^{\frac25}f_1)|+|(3Q_a^2 f_1,\H_a^{\frac25} f_1)|+|(\lambda f_2,\H_a^{\frac25}f_1)|+|(F_2,\H_a^{\frac25}f_1)|\\
&\lesssim \|f_1\|_{H^1}^2 + \|Q_a^2\|_{L^3} \|f_1\|_{H^1}^2 + (\lambda \|f_2\|_{L^2} + \|F_2\|_{L^2}) \|f_1\|_{H^1}\lesssim 1.
\end{align*}
To complete the argument, we use the Sobolev embedding $f\in \dot H^{1+\frac25}\hookrightarrow \dot H^{1,\frac{30}{11}}.$\end{proof}

\section{Existence and uniqueness of particular solutions}\label{S:Special}

In this section, we describe the construction of particular solutions $U^A(t)$ (indexed by $A\in\R$) to \eqref{NLSa} that converge exponentially to the ground state solution in $H^1$ as $t\to\infty$. The heteroclinic orbits described in Theorem~\ref{T} ultimately correspond to the choices $A=\pm 1$.  The strategy, adapted from works such as \cite{DM}, is to begin by building a sequence of approximate solutions to the linearized equation and then to apply a fixed point argument.  In fact, given the spectral results of Section~\ref{S:Spectral}, the equivalence of Sobolev spaces, and Strichartz estimates for $e^{-it\mathcal{H}_a}$, most of the arguments parallel those appearing in \cite{DM,CM2023} quite closely.  Thus, in order to streamline our presentation, our approach will be to state the main results we need and refer the reader to \cite{DM,CM2023} for the details of the proofs. 

We begin by recording some basic nonlinear estimates that rely on the bounds on $Q_a$ given in Proposition~\ref{boundsGS}. Recall that the norms $N(I)$ and $Z(I)$ were defined in Section~\ref{S:notation}, and that the notation $K(f)$ and $R(f)$ was introduced in \eqref{K(v)R(v)}.  For further details, we refer the reader to \cite[Lemma~5.6]{DM} (or \cite[Lemma~7.1]{CM2023}).

\begin{lem} \label{nonlinear} If $I \subset \mathbb{R}$ is such that $|I| \leq 1$, then we have the estimates
\begin{enumerate}
\item[(i)] $\|K(f)\|_{N(I)} \lesssim |I|^{\frac12} \|f\|_{Z(I)}$, \\
\item[(ii)] $\| R(f) - R(g)\|_{N(I)} \lesssim \|f-g\|_{Z(I)} (\|f\|_{Z(I)} + \|g\|_{Z(I)} + \|f\|_{Z(I)}^2 + \|g\|_{Z(I)}^2)$.
\end{enumerate}
\end{lem}


We turn to the construction of approximate solutions. We recall the linearized operator $\mathcal{L}$ defined in \eqref{linearized}.   The proof of the following proposition is an inductive argument; we refer the reader to \cite[Lemma~6.1]{DM} (or \cite[Proposition~7.2]{CM2023}) for complete details.  We note that the proof makes use of the decay estimates for the eigenfunctions obtained in Proposition~\ref{F_reg}.

\begin{prop} \label{approximatesol}
Let $A \in \mathbb{R}\backslash \{0\}$. There exists a sequence $\{ Z_k\}_{k \geq 1}$ of functions in $L_t^\infty H^2_a \cap L_t^\infty H^{1, {\frac{30}{11}}}$ such that
\[
Z_1 = A Y_+ \qtq{and} V_k = \sum_{j=1}^k e^{-je_0t} Z_j
\]
satisfy
\begin{equation}\label{Vk-error}
\partial_t V_k + \mathcal{L} V_k = iR(V_k) + \mathcal{O}(e^{-(k+1)e_0t}) \qtq{in} H^{1, \frac{10}{7}} \qtq{as} t \rightarrow\infty.
\end{equation}
\end{prop}

In the proof, the functions $Z_k$ are constructed inductively.  For example, one first computes that $R(V_1)$ is written as a sum of terms of the form $e^{-2e_0 t}F_1$ and $e^{-3e_0t}F_2$. The function $Z_2$ is then defined as the solution to  $(\mathcal{L}-2e_0)Z_2=-F_1$ in order to cancel the $\mathcal{O}(e^{-2e_0t})$ contribution. Similarly, $Z_k$ is constructed as the solution to an equation of the form $(\mathcal{L}-ke_0)Z_k = -F_k$ in order to cancel terms of order $\mathcal{O}(e^{-ke_0t})$.

With the approximate solutions in hand, we may now apply a fixed point argument to obtain a true solution.  Details may be found in \cite[Proposition~6.3]{DM} (or \cite[Proposition~7.3]{CM2023}). 

\begin{prop} \label{spe}
Let $A \in \mathbb{R} \backslash \{0\}$. There exists $k_0>0$ sufficiently large such that the following holds: for any $k\geq k_0$, there exists $t_k\geq 0$ and a solution $U^A$ to \eqref{NLSa} such that 

\begin{equation} \label{eq:7.9}
\|U^A - e^{it}Q_a - e^{it}V_k^A\|_{Z([t,\infty))} \leq e^{-(k+\frac{1}{2})e_0 t} \qtq{for}t\geq t_k,
\end{equation}
where $V_k^A$ are as in Proposition~\ref{approximatesol}.

Furthermore, $U^A$ is the unique solution to \eqref{NLSa} satisfying \eqref{eq:7.9} for large $t$.

Finally, $U^A$ is independent of $k$ and satisfies
\begin{equation} \label{eq:7.10}
\|U^A(t) - e^{it}Q_a - Ae^{-c_0 t + it} Y_+\|_{H^1} \leq e^{-2e_0 t}\qtq{for large}t.
\end{equation}
\end{prop}

We next consider uniqueness for solutions converging to the ground state solution. {The key ingredient is the following proposition.}  For details of the proof, we refer the reader to \cite[Proposition~5.9]{DM} (or \cite[Proposition~8.1]{CM2023}). 

\begin{prop}\label{aux_prop} Suppose $h$ is a solution to $\partial_t h + \mathcal{L} h = \ep$ satisfying
\[
\| h(t) \|_{H^1} \lesssim e^{-c_1 t}\qtq{and} \| \ep \|_{N([t, \infty))} \lesssim e^{-c_2 t}
\]
for some $0 < c_1 < c_2$. Then there exists $A \in \mathbb{R}$ such that
\begin{equation}\label{auxprop}
\| h(t) - A e^{-e_0 t} Y_+ \|_{H^1} \lesssim e^{-\frac{c_1 + c_2}{2} t}
\end{equation}
for all $t > 0$. Moreover, if $c_1 > e_0$ or $c_2 \le e_0$, then we may choose $A = 0$.
\end{prop}

Following the proof of \cite[Proposition~8.2]{CM2023}, we can now obtain the following uniqueness result.

\begin{prop} \label{uniqui} Suppose $u$ is a solution to \eqref{NLSa} satisfying
\[
\| u(t) - e^{it}Q_a\|_{H^1} \lesssim e^{-ct}
\]
for some $c>0$ and all large $t$. Then there exists $A \in \mathbb{R}$ such that $u = U^A$. 
\end{prop}

We now record two useful corollaries. 

\begin{cor}\label{C:YouAreSpecial} If $u$ is a solution to \eqref{NLSa} satisfying
\[
\| u(t) - U^A(t)\|_{H^1} \lesssim e^{-ct}\qtq{for large}t
\]
for some $c > e_0$ and $A \in \mathbb{R}$, then $u = U^A$.
\end{cor}

\begin{proof} 
Observe that
\begin{align*}
\|u(t) - e^{it}Q \|_{H^1} &\lesssim \| u(t) - U^A(t) \|_{H^1} + \| V^A(t)\|_{H^1} \\
&\lesssim e^{-ct} + e^{-e_0 t}.
\end{align*}
Therefore, by Proposition \ref{uniqui}, there exists $A_0$ such that $u = U^{A_0}$. Moreover,
\[
\|U^A(t) - U^{A_0}(t)\|_{H^1} = \| V^A(t) - V^{A_0}(t) \|_{H^1} \sim |A - A_0| e^{-e_0 t} + \mathcal{O}(e^{-2e_0 t}).
\]
As $c > e_0$, we conclude that $A = A_0$.
\end{proof}

\begin{cor}\label{C:YouAreMoreSpecial}
For any $A>0$, there exists $T_A$ such that 
\[
U^A(t) = e^{-iT_A} U^{+1}(t + T_A).
\]
For $A<0$ the same statement holds with $U^{-1}$.
\end{cor}

\begin{proof} We estimate
\begin{align*}
\| e^{iT_A} U^A(t) - U^{+1}(t+ T_A) \|_{H^1}
&= \| e^{iT_A} U^A(t) - e^{i(t+ T_A)} \big( Q_a + e^{-e_0(t+T_A)} Y_+ \big) \|_{H^1}  \\ &
\quad\quad  + \mathcal{O}(e^{-2e_0 t}) \\
&= \| U^A(t) - e^{it}(Q_a + A e^{-e_0 t} Y_+) \|_{H^1} + \mathcal{O}(e^{-2e_0 t}) \\
&= \mathcal{O}(e^{-2e_0 t}),
\end{align*}
where we choose $T_A = -\tfrac{\ln A}{e_0}.$ The previous corollary then yields the result. Arguing analogously yields the result for $A<0$ as well. \end{proof}

In Section~\ref{S:Proof}, we will show that $Q_a^\pm:= U^{\pm 1}$ are the heteroclinic orbits described in Theorem~\ref{T}.

\section{Modulation Analysis}\label{S:Modulation}

The goal of this section is to analyze solutions when they are close to the orbit of the ground state, as quantified by the functional
\[
d(u(t)) \coloneq \bigl|\, \| u(t) \|_{\dot{H}^1_a}^2 - \| Q_a \|_{\dot{H}^1_a}^2 \bigr|.
\]
Most of the analysis in this section closely parallels arguments appearing in works such as \cite{DM, CM2023}.  Thus we will omit details for several of the results, referring the reader instead to the corresponding results in these works.

We begin with the following lemma. 

\begin{lem}\label{aux_mod} For any $\ep>0$, there exists $\delta>0$ such that if $v\in H^1_{\text{rad}}$ satisfies 
\[
M(v)=M(Q_a),\quad E_a(v)=E_a(Q_a), \qtq{and}d(v)<\delta,
\]
then
\[
\inf_{\theta\in\R} \|v-e^{i\theta}Q_a\|_{H_a^1}<\ep. 
\]
\end{lem}

\begin{proof} Suppose the statement were false. Then there exists $\ep_0>0$ and $\{f_n\}\subset H_{\text{rad}}^1$ satisfying $M(f_n)=M(Q_a)$, $E_a(f_n)=E_a(Q_a)$, $d(f_n)\to 0$, and
\begin{equation}\label{nope}
\inf_{\theta\in\R}\|f_n-e^{i\theta}Q_a\|_{H_a^1}\geq\ep_0. 
\end{equation}
In particular, it follows that
\begin{equation}\label{fnQa}
\|f_n\|_{L^4} \to \|Q_a\|_{L^4}. 
\end{equation}

As $f_n$ is bounded in $H^1$, we may pass to a subsequence to obtain a weak limit $\phi\in H^1$.  By the compactness of the embedding $H^1_{\text{rad}}\hookrightarrow L^4$, it follows that $f_n\to\phi$ strongly in $L^4$. By weak lower-semicontinuity followed by the mass and energy constraints and \eqref{fnQa}, we can derive that
\begin{align*}
C_a \|\phi\|_{L^2}\|\phi\|_{\dot H^1_a}^3 \leq C_a \|Q_a\|_{L^2}\|Q_a\|_{\dot H_a^1}^3 = \|Q_a\|_{L^4}^4 = \|\phi\|_{L^4}^4.
\end{align*}
It follows that $\phi$ is an optimizer for the Gagliardo--Nirenberg inequality (whence $\phi=e^{i\theta}Q_a$), and moreover that the weak $H^1$ convergence may be upgraded to strong convergence.  In particular, $f_n\to e^{i\theta}Q_a$ in $H^1$, contradicting \eqref{nope}.\end{proof}

We turn to the main result of this section.

\begin{prop} \label{modulationprop}
Let $u: I \times \mathbb{R}^3 \rightarrow \mathbb{C}$ be a solution to \eqref{NLSa} with $M(u_0) = M(Q_a)$ and $E_a(u_0) = E_a(Q_a)$, and define $I_0 = \{t \in I : d(u(t)) < \delta_0\}$. If $\delta_0$ is sufficiently small, then there exist functions $\alpha , \theta : I_0 \rightarrow \mathbb{R}$ and $h : I_0 \rightarrow H^1_a$ such that
\begin{equation}\label{uuuqqqhhh}
u(t) = e^{i \theta(t)}[(1+ \alpha(t))Q_a + h(t)]
\end{equation}
with
\[
|\alpha(t)| \sim \|h(t)\|_{H^1_a} \sim d(u(t))\qtq{and} | \alpha'(t)| + | \theta'(t) -1 | \lesssim d(u(t)).
\]
\end{prop}

The proof of this proposition largely follows the arguments of \cite[Lemma~3.6]{DM} and \cite[Lemma~3.7]{DM} (or \cite[Proposition~4.1]{CM2023}). The basic idea is to obtain an initial decomposition as in Lemma~\ref{aux_mod} and then use the Implicit Function Theorem to obtain modulation parameters that impose desired orthogonality conditions.  In particular, one defines $\theta(t)$ to impose $\Im(u(t),e^{i\theta(t)}Q_a)\equiv 0$ and then defines
\[
\alpha(t)=\frac{(\Re[e^{-i\theta(t)}u(t)-Q_a],\H_a Q_a)}{(Q_a,\H_a Q_a)}.
\]
The function $h(t)$ is then defined by \eqref{uuuqqqhhh} and satisfies $(\Im h,Q_a)=(\Re h,\H_a Q_a)=0$. These conditions imply coercivity for the operators $L_\pm$ (cf. Proposition~\ref{coercivity} above), which play a key role in obtaining the bounds appearing in the proposition. 

We turn to the following corollary, which follows from the arguments as in \cite[Corollary~4.3]{DM} or \cite[Proposition~3.1]{CM2023}. 

\begin{cor} \label{convofdistance}
Let $u$ be a global forward-in-time solution to \eqref{NLSa} with $M(u) = M(Q_a)$ and $E(u) = E(Q_a)$. If \begin{align}\label{intdistance}
\int_0^\infty d(u(s))\, ds < \infty,
\end{align}
then $\lim_{t\to\infty } d(u(t)) = 0$, and there exists $\theta_0 \in \mathbb{R}$ such that 
\[
\|u(t) - e^{i(t+ \theta_0)} Q_a\|_{H^1_a} \lesssim \int_t^\infty d(u(s))\, ds
\]
for all $t$ sufficiently large.
\end{cor}

\section{Convergence for nonscattering constrained solutions}\label{S:Constrained}

In this section we seek to prove that any constrained solution which fails to scatter must converge exponentially to the ground state solution. The first step in proving this goal is a compactness result for non-scattering constrained solutions.

\begin{prop}\label{P:Precompactness} Let $a>0$.
Suppose $u_0\in H_{\text{rad}}^1(\R^3)$ satisfies $M(u_0)=M(Q_a)$, $E(u_0)=E(Q_a)$, and $\|u_0\|_{\dot{H}_a^1}<\|Q_a\|_{\dot{H}_a^1}$. Let $u:\R\times\R^3\to \mathbb{C}$ be the global solution to \eqref{NLSa}. If
\begin{equation}\label{P:PSInfnorm}
\|u\|_{L_{t,x}^5([0,\infty))}=\infty,
\end{equation}
then $\{u(t):t\in [0,\infty)\}$ is pre-compact in $H^1$.  A similar statement holds backwards in time.
\end{prop}

\begin{proof} The arguments needed to prove Proposition~\ref{P:Precompactness} are by now fairly standard.  This is especially true in the radial setting, in which no translation parameters appear in the linear profile decomposition (and hence the additional challenges due to the broken translation symmetry of \eqref{NLSa} vanish completely). Therefore, we will proceed by giving a fairly brief summary of the arguments.

Let $\{\tau_n\}\subset [0,\infty)$ be a sequence of times. By continuity of the flow, we may assume that $\tau_n\to \infty$. Applying the linear profile decomposition, Proposition~\ref{P:LPD}, we write for any finite $J\leq J^*$
\[
u(\tau_n)=\sum_{j=1}^J\phi_n^j+r_n^J,\qtq{with}\phi_n^j=e^{it_n^j\H_a}\phi^j(x)
\]
along a subsequence with all the properties detailed in Proposition~\ref{P:LPD}. We will show that $J^*=1$, $t_n^1\equiv 0$, and that the remainder $r_n^J$ vanishes strongly in $H^1$ as $n\to\infty$, which will yield the result. To this end, we will show that all other scenarios contradict \eqref{P:PSInfnorm}.

We can quickly rule out the case $J^*=0$, for in that case an application of the stability theory with the approximate solutions $e^{-it\H_a}u(t_n)$ implies space-time bounds for $u$ on $[t_n,\infty)$ for all $n$ large, contradicting \eqref{P:PSInfnorm}.


It follows that $J^*\geq 1$. Now, the decoupling statement in the linear profile decomposition and Proposition~\ref{P:Coercivity} imply
\[
\limsup_{n\to\infty}\|\phi_n^j\|_{L_x^2}\|\phi_n^j\|_{\dot{H}_a^1}\leq\mathcal{K}_a.
\]
Thus the sharp Gagliardo--Nirenberg inequality implies that $\liminf_n E_a(\phi_n^j)\geq 0$ for all $j$. Similar considerations show that $\liminf_{n} E_a(r_n^J)\geq 0$ for any finite $J$. Thus there are two possible cases:
\begin{align}
&\text{Case 1:}\quad\sup_j\limsup_n M(\phi_n^j)E_a(\phi_n^j)=\mathcal{E}_a,\label{E:PSJ*=1}\\
&\text{Case 2:}\quad \sup_j\limsup_n M(\phi_n^j)E_a(\phi_n^j)\leq \mathcal{E}_a-3\delta\qtq{for some} \delta>0.\label{E:PSJ*G1}
\end{align}

We now show that \eqref{E:PSJ*G1} is inconsistent with \eqref{P:PSInfnorm}. Arguing as above, we see that for any $j$ we have that 
\[
\|\phi_n^j\|_{L_x^2}\|\phi_n^j\|_{\dot{H}_a^1}< (1-\delta')\mathcal{K}_a,
\]
for $n$ sufficiently large. We associate to each $\phi^j$ a corresponding scattering solution to \eqref{NLSa}. If $t_n^j\equiv0$ we define $v^j=v_n^j$ as the global scattering solution to \eqref{NLSa} with initial data $\phi^j$. If instead $t_n^j\to \pm\infty$ we define $v^j$ as the global solution to \eqref{NLSa} that scatters to $\phi^j$ as $t\to\pm\infty$, and then define $v_n^j(t,x)=v^j(t+t_n^j,x)$. 

We then define an approximate solution to \eqref{NLSa} by
\[
u_n^J=\sum_{j=1}^Jv_n^j+e^{it\H_a}r_n^J.
\]
By construction we have $u_n^J(0)-u(t_n)\to 0$ in $H^1$ as $n\to\infty$. We claim that 
\begin{align}
&\limsup_{J\to J^*}\limsup_{n\to\infty}\{\|u_n^J(0)\|_{H_a^1}+\|u_n^J\|_{L_{t,x}^5}\}\lesssim 1,\label{E:PS_unJ_goodbounds}\\
&\limsup_{J\to J^*}\limsup_{n\to\infty}\||\nabla|^\frac{1}{2}e_n^J\|_{L_{t,x}^\frac{10}{7}}=0, \label{E:PS_unJ_errorbounds}
\end{align}
where the error $e_n^J$ is defined by
\[
e_n^J=(i\partial_t-\H_a)u_n^J+|u_n^J|^2u_n^J.
\]
In this case, we can apply Proposition~\ref{P:Stab} and obtain a contradiction to \eqref{P:PSInfnorm}.

Both \eqref{E:PS_unJ_goodbounds} and \eqref{E:PS_unJ_errorbounds} rely on the following orthogonality result, which follows from \eqref{tnjorthogonal} and approximation by $C_c^\infty$ functions. 
\begin{lemma}[Orthogonality, \cite{KVMZ} Lemma 7.4]\label{L:PSOrthog}
For any $j\not=k$ we have
\[
\|v_n^jv_n^k\|_{L_{t,x}^\frac{5}{2}}+\|v_n^jv_n^k\|_{L_{t,x}^\frac{5}{3}}+\|v_n^jv_n^k\|_{L_{t}^\frac{15}{7}L_x^\frac{45}{31}}+\|(\H_a)^\frac{31}{120}v_n^j(\H_a)^\frac{31}{120}v_n^k\|_{L_t^\frac{15}{7}L_x^\frac{45}{31}}\to 0
\]
as $n\to\infty.$
\end{lemma}

%
%
%
%

%
Using Lemma~\ref{L:PSOrthog} and the space-time bounds obeyed by each individual solution $v_n^j$, we can firstly obtain \eqref{E:PS_unJ_goodbounds}.  To estimate the error and obtain \eqref{E:PS_unJ_errorbounds}, we write $F(z)=-|z|^2z$ and decompose $e_n^J$ as follows
\begin{align}
e_n^J&=\sum_{j=1}^JF(v_n^j)-F\bigl(\sum_{j=1}^Jv_n^j\bigr)\label{E:PS_error_term1}\\
&\quad+F(u_n^J-e^{-it\H_a}r_n^J)-F(u_n^J)\label{E:PS_error_term2}
\end{align}
For \eqref{E:PS_error_term1}, we can use pointwise bounds and Lemma~\ref{L:PSOrthog} to obtain vanishing in $L_{t,x}^{\frac{10}{7}}$.  Interpolating this with boundedness in the presence of slightly more than $\tfrac12$ derivatives (which follows from the global space-time bounds of the individual solutions $v_n^j$), we obtain the desired estimate.  For \eqref{E:PS_error_term2}, we similarly use interpolation, this time relying on \eqref{rnjvanish} to obtain vanishing with no derivatives.

At this point, we have shown that \eqref{E:PSJ*=1} holds. In particular, using 
\[
\liminf_{n}E_a(\phi_n^j)\geq 0\qtq{and}\liminf_{n}E_a(r_n^J)\geq 0,
\]
we derive that $J^*=1$, and so
\[
u(\tau_n)=\phi_n^1+r_n^1,\qtq{with}r_n^1\to 0\qtq{strongly in}H^1.
\]
To conclude the argument, it remains to see that $t_n^1\equiv 0$.  Indeed, if instead $t_n^1\to\pm\infty$, then approximation by free solutions contradicts \eqref{P:PSInfnorm}. \end{proof}

We now establish that forward non-scattering constrained solutions must converge exponentially to the ground state solution as $t\to\infty$. 

\begin{prop} \label{compact_constrained} Let $a>0$. Suppose $u$ is a global radial solution to \eqref{NLSa} satisfying 
\[
M(u) = M(Q_a), \qquad E_a(u) = E_a(Q_a), \qquad \|u(0)\|_{\dot{H}^1_a} < \|Q_a\|_{\dot{H}^1_a},
\]
and
\[
\|u\|_{L_{t,x}^5([0,\infty)\times\R^3)}=\infty.
\]

Then there exist constants $C, c > 0$ and $\theta_0 \in \mathbb{R}$ such that 
\[
\|u(t) - e^{i(t+\theta_0)}Q_a\|_{H^1} \leq C e^{-ct}\qtq{for all}t\geq 0.
\]
\end{prop}

In order to prove this result, we begin with a virial estimate that incorporates the modulation analysis of Section 6.

\begin{lemma}\label{virialscat} Let $u$ be as in Proposition~\ref{compact_constrained}.  There exists a constant $C>0$ such that for any $[t_1, t_2] \subset [0, \infty)$ the following estimate holds:
\[
\int_{t_1}^{t_2} d(u(t))\, dt \leq C\bigl(d(u(t_1)) + d(u(t_2))\bigr).
\]
\end{lemma}
\begin{proof} The proof follows along the lines of previous works such as \cite{DM, DR ,CM2023}, so we will keep our presentation brief.  We rely on the following virial identity (see \cite[Section 2.4]{KVMZ}): Let $R \in [1,\infty]$ and let $\phi$ be a real-valued radial function satisfying
\begin{align*}
\phi(x)=\begin{cases} |x|^2 & |x|\leq 1\\ \text{const.}& |x|>3,
\end{cases}\qtq{with} |\nabla \phi|\leq 2|x|,\quad |\partial_{jk}\phi|\leq 2.
\end{align*}
We then define
\begin{equation}\label{E:VirialQuantity}
P_R[u]=2\,\mathrm{Im}\int \overline{u}\,\nabla u\cdot\nabla w_R\,dx,
\end{equation}
where $w_R(x)=R^2\phi\!\left(\frac{x}{R}\right)$ and $w_\infty(x):=|x|^2$. We have the following identities:
\[
\tfrac{d}{dt} P_R[u(t)] = F_R[u(t)],
\]
where for finite $R$ we have
\[
F_R[u]=\int [(-\Delta^2 w_R)|u|^2
+4\,\mathrm{Re}\,\overline{u}_j u_k\,\partial_{jk}(w_R) +4a|u|^2\tfrac{x}{|x|^4}\cdot\nabla w_R
-|u|^4\Delta w_R]\,dx
\]
and (using the energy constraint and \eqref{pohozaev})
\[
F_\infty[u] = 8\bigl[\|u\|_{\dot H^1_a}^2-\tfrac{3}{4}\|u\|_{L^4}^4\bigr]
=4\,d(u(t)).
\]

We note that $P_R[e^{i\theta}Q_a]=0$ and $F_R[ e^{i \theta}Q_a] = 0$ for all $R \in [1, \infty]$ and $\theta \in \mathbb{R}$ (see e.g. \cite{CM2023, MiaoMurphyZheng2023}).  

Let $R \geq 1$ (which will be chosen large enough below) and $\delta_1 \in (0, \delta_0)$, where $\delta_0$ is as in Proposition~\ref{modulationprop}. We define the characteristic function 
\[
\chi = \chi_J,\quad J = \{ t \in [t_1,t_2] : d(u(t)) < \delta_1 \},
\]
and its complement $\chi^c = 1- \chi$. We write
\begin{align} \label{auxscat}
\tfrac{d}{dt} P_R[u] - F_\infty[u] &=  \chi^c(t) (F_R[u] - F_\infty[u]) \nonumber\\
&\quad + \chi(t) \bigl(F_R[u]-F_\infty[u] - (F_R[e^{i \theta(t)}Q_a]- F_\infty[e^{i\theta(t)}Q_a])\bigr).
\end{align}

We now estimate $P_R[u(t_j)]$ as well as the right-hand side of this equality.  We begin by fixing $j \in \{1,2\}$. If $d(u(t_j)) \geq \delta_1$ then 
\[
|P_R[u(t_j)] | \lesssim R \|u\|^2_{L^\infty_t H^1_x} \lesssim \tfrac{R}{\delta_1} d(u(t_j)).
\]
If instead $d(u(t_j)) < \delta_1$, we use the modulation analysis and estimate as follows: 
\begin{align*}
|P_R[u(t_j)] | &= \biggl| 2\,\Im \int \bigl( \bar{u}\nabla u - e^{-i \theta(t_j)}Q_a \nabla [e^{i\theta(t_j)}Q_a]\bigr) \cdot \nabla w_R \,dx \biggr| \\
&\lesssim R\bigl(\| \bar{u}\nabla (u(t_j)- e^{i\theta(t_j)}Q_a)\|_{L^1} + \|(\bar{u}- e^{-i\theta(t_j)}Q_a) \nabla e^{i\theta(t_j)}Q_a \|_{L^1}\bigr) \\
&\lesssim R\bigl(\|u\|_{L^2} + \| \nabla Q_a\|_{L^{\frac65}}\bigr) d(u(t_j)) \\
&\lesssim R \,d(u(t_j)).
\end{align*}

To estimate the right-hand side we begin with the $\chi^c$ term. As $d(u(t)) \geq \delta_1$ on the support of $\chi^c$, we rely directly on $H^1$ precompactness. We begin by writing
\begin{align*}
F_R[u] - F_\infty[u] = &\int_{|x|>R} (-\Delta^2 w_R)|u|^2 + 4\,\mathrm{Re}\, \overline{u_j} u_k\, \partial_{jk}(w_R) + 4a |u|^2 \tfrac{x}{|x|^4} \cdot\nabla w_R \\
&-|u|^4 \Delta w_R \,dx - 8\|u(t)\|_{\dot{H}^1_a(|x|>R)}^2 + 6 \|u(t)\|^4_{L^4(|x|>R)}.
\end{align*}
Now, given $\ep>0$, we choose $R>1$ large enough such that all the above integrals are bounded by $\ep$ uniformly for $t \in [t_1,t_2]\backslash J$. Thus, we find
\[
\bigl|\chi^c(t)[F_R[u] - F_\infty[u]] \bigr| < \tfrac{\ep}{\delta_1} d(u(t)).
\] 

Next, we write $Q_a(t) = e^{i\theta(t)}Q_a$ and expand the term involving $\chi$ as
\begin{align*}
&-8\bigl[\|u(t)\|_{\dot{H}^1_a(|x|>R)}^2-\|Q_a(t)\|_{\dot{H}^1_a(|x|>R)}^2\bigr] + 6 \bigl[\|u(t)\|^4_{L^4(|x|>R)}-\|Q_a(t)\|^4_{L^4(|x|>R)}\bigr] \\
&\quad + \int_{|x|>R} 4\,\Re\bigl[\overline{u_j}u_k -\overline{Q_a(t)}_jQ_a(t)_k\bigr] \partial_{jk}(w_R) - \bigl[|u|^2-|Q_a(t)|^2\bigr]\Delta^2w_R \,dx \\
&\quad - \int_{|x|>R} \bigl[|u|^4 - |Q_a(t)|^4\bigr] \Delta w_R -4a\bigl[|u|^2-|Q_a(t)|^2\bigr] \tfrac{x}{|x|^4} \cdot \nabla w_R \,dx.
\end{align*}
In each term we can exhibit the difference $u(t)-Q_a(t)$, which we estimate in $H^1$ by $d(u(t))$ (cf. Proposition~\ref{modulationprop}).  The remaining terms contain either $u$ or $Q_a$ norms on $|x|>R$. By the compactness, we can take $R$ large enough such that these norms are all $\mathcal{O}(\ep)$. For example, 
\begin{align*}
\left| \int_{|x|>R} \bigl[|u|^4- |Q_a(t)|^4\bigr] dx \right| &\lesssim \int_{|x|>R} (|u|^3 + |Q_a(t)|^3) |u-Q_a(t)| \,dx \\
&\lesssim\bigl(\|u\|^3_{L^4(|x|>R)} + \|Q_a(t)\|^3_{L^4(|x|>R)}\bigr) \|u-Q_a(t)\|_{L^4} \\
&\lesssim \ep\, d(u(t)).
\end{align*}
In this way we obtain 
\[
\bigl|  \bigl(F_R[u]-F_\infty[u] - (F_R[e^{i \theta(t)}Q_a]- F_\infty[e^{i\theta(t)}Q_a])\bigr) \bigr| \lesssim \ep\, d(u(t))\qtq{for}t\in J.
\]

We now integrate \eqref{auxscat} over $[t_1,t_2]$, using  $F_\infty[u] = 4d(u(t))$, $\delta_1 \ll 1$, and the estimates above to obtain 
\[
\int_{t_1}^{t_2} d(u(t))\, dt  \leq \tfrac{CR}{\delta_1} \bigl[d(u(t_1)) + d(u(t_2))\bigr] + \tfrac{\ep}{\delta_1} \int_{t_1}^{t_2} d(u(t))\, dt.
\]
Choosing $\ep=\ep(\delta_1)$ small enough we obtain the desired estimate.  \end{proof}

Applying the previous lemma on the interval $[0,T]$ and using uniform boundedness of $\delta(u(t))$, we obtain:

\begin{cor}\label{finiteint} Let $u$ be as in Proposition~\ref{compact_constrained}. Then
\[
\int_0^{\infty} d(u(t)) \, dt < \infty.
\]
\end{cor}
We now have all the ingredients to prove Proposition \ref{compact_constrained}.

\begin{proof}[Proof of Proposition \ref{compact_constrained}.] Combining Lemma~\ref{virialscat}, Corollary~\ref{finiteint}, and Corollary~\ref{convofdistance}, we obtain 
\[
\int_t^\infty d(u(s)) \, ds \leq C d(u(t))
\]
for all $t>0$. Thus Gronwall's inequality implies 
\[
\int_t^\infty d(u(s)) \, ds \lesssim e^{-ct}
\]
for some $c>0$ and all $t>0$. Finally, we apply Corollary~\ref{convofdistance} to obtain
\[
\|u(t) - e^{i(t+ \theta_0)} Q_a\|_{H^1} \lesssim e^{-ct}
\]
for some $\theta_0\in\R$ and all $t>0$.\end{proof}

\section{Convergence for global unconstrained solutions}\label{S:Unconstrained}

In this section we prove an analogue of Proposition~\ref{compact_constrained} in the case of a forward global \emph{unconstrained} solution.  In this case, we cannot obtain compactness.  Instead we will rely heavily on radiality through the radial Gagliardo--Nirenberg inequality (see Lemma~\ref{L:RadialGNI}). 

\begin{prop}\label{P:UncSolConv}
Suppose $u$ is a radial, forward-global solution to \eqref{NLSa} with 
\[
M(u)=M(Q_a),\quad E_a(u)=E_a(Q_a),\quad \|u_0\|_{\dot{H}_a^1}>\|Q_a\|_{\dot{H}_a^1}.
\]
Then there exists $c>0$ and $\theta_0\in \R$ such that
\[
\|u(t)-e^{i(t+\theta_0)}Q_a\|_{H^1}\lesssim e^{-ct}\qtq{for all}t\geq 0.
\]
\end{prop}

Recalling the notation from Section~\ref{S:Constrained}, we first prove an analogue of Lemma~\ref{virialscat}.
\begin{lemma}\label{L:UncVir} Let $u$ be as in Proposition~\ref{P:UncSolConv}. There exists $R$ sufficiently large such that for $[t_1, t_2] \subset [0, \infty)$,
\begin{equation}\label{E:UncVir}
\int_{t_1}^{t_2} d(u(t))\,dt\leq C\bigl(P_R[u(t_1)]-P_R[u(t_2)]\bigr).
\end{equation}
Furthermore,
\begin{equation}\label{E:UncVir2}
\int_t^\infty d(u(s))\,ds\lesssim d(u(t)).
\end{equation}
\end{lemma}

\begin{proof} We use the same notation for the virial quantities used in Lemma~\ref{virialscat}. We fix an interval $[t_1,t_2]$ and $\delta_1\in (0,\delta_0)$, where $\delta_0$ is as in Proposition~\ref{modulationprop}. We define
\[
J=\{t\in [t_1,t_2]:d(t)<\delta_1\},\quad \chi=\chi_J,\quad\chi^c=1-\chi.
\]
As in Lemma~\ref{virialscat}, we have
\begin{align}
\tfrac{d}{dt}P_R[u]-F_\infty[u]&=\chi^c(t)\{F_R[u]-F_\infty[u]\}\label{E:UncVirialMain1}\\
&\quad+\chi(t)\{F_R[u]-F_\infty[u]-(F_R[e^{i\theta(t)}Q_a]-F_\infty[e^{i\theta(t)}Q_a])\}\label{E:UncVirialMain2}.
\end{align}

For \eqref{E:UncVirialMain1}, we write
\begin{align}
F_R[u]-F_\infty[u]&= \int_{|x|>R}4[\Re\bar{u}_j\partial_{jk}w_Ru_k]-8|\nabla u|^2\,dx\label{E:UncVir1.1}\\
&\quad +\mathcal{O}\biggl(\int_{|x|>R}R^{-2}|u|^2+|u|^4\,dx\biggr)\label{E:UncVir1.2}.
\end{align}
The properties of the virial weight $\phi$ guarantee that $\eqref{E:UncVir1.1}\leq 0$. For the first term in \eqref{E:UncVir1.2} we use the support properties of $\chi^c$ to estimate
\[
R^{-2}\|u\|_{L^2}^2\lesssim R^{-2}\tfrac{d(t)}{\delta_1}.
\]
For the second term in \eqref{E:UncVir1.2} we instead use the radial Gagliardo-Nirenberg inequality Lemma~\ref{L:RadialGNI} to obtain
\begin{align*}
\|u\|_{L^4(|x|>R)}^4&\lesssim R^{-2}\|u\|_{\dot{H}_a^1}\|u\|_{L_x^2}^3\\
&\lesssim R^{-2}\sqrt{d(u(t))+1}\lesssim R^{-2}d(u(t))(\delta_1^{-1}+\delta_1^{-2})^{\frac12}
\end{align*}
Thus
\[
|\eqref{E:UncVirialMain1}|\lesssim \tfrac{1}{\delta_1R^2}d(u(t)).
\]
We turn to \eqref{E:UncVirialMain2}.  We first write
\begin{align}
F_R[u]&-F_\infty[u]-(F_R[e^{i\theta(t)}Q_a]-F_\infty[e^{i\theta(t)}Q_a])\nonumber\\
&=\int_{|x|>R}4\text{Re}[\partial_{jk}w_R(\overline{u}_ju_k-\partial_jQ_a\partial_kQ_a)]-8[|\nabla u|^2-|\nabla Q_a|^2]\,dx\label{E:UncVir2.1}\\
&\quad+\int_{|x|>R}4[|u|^2-|Q_a|^2]\tfrac{ax}{|x|^4}\cdot\nabla w_R-8[|u|^2-|Q_a|^2]\tfrac{a}{|x|^2}\,dx\label{E:UncVir2.2}\\
&\quad+\mathcal{O}\biggl(\int_{|x|>R}R^{-2}(|u|^2-|Q_a|^2)+(|u|^4-|Q_a|^4)\,dx\biggr).\label{E:UncVir2.3}
\end{align}

For $t\in J$ we may use Proposition~\ref{modulationprop} to decompose
\[
u(t)=e^{i\theta(t)}[(1+\alpha(t))Q_a+h(t)].
\]
Writing $Q_a(t) = e^{i\theta(t)}Q_a$, we then estimate \eqref{E:UncVir2.1} as follows: 
\begin{align*}
|\eqref{E:UncVir2.1}|&\lesssim\int_{|x|>R}(2|\alpha|+|\alpha|^2)|\nabla Q_a|^2+(1+|\alpha|)|\nabla h|\,|\nabla Q_a(t)|+|\nabla h|^2\,dx\\
&\lesssim d(u(t))\|Q_a\|_{H_x^1(|x|>R)}^2+(1+d(u(t)))d(u(t))\|Q_a\|_{H_x^1(|x|>R)}+[d(u(t))]^2\\
&\lesssim (o_R(1)+\delta_1)d(u(t)),
\end{align*}
where $o_R(1) = \|Q_a\|_{H^1(|x|>R)}\to 0$ as $R\to\infty$. We can similarly obtain
\begin{align*}
|\eqref{E:UncVir2.2} | + | \eqref{E:UncVir2.3}|&\lesssim R^{-2}\int_{|x|>R}(2|\al|+|\al|^2)|Q_a|^2+(1+|\al|)|h|\,|Q_a|+|h|^2\,dx\\
& \quad + \int_{|x|>R} |\alpha| |Q_a|^4 + |h|^4 + |Q_a|^3 |h|\,dx \\
&\lesssim o_R(1) d(u(t))+[d(u(t))]^2 + [d(u(t))]^4 \\
&\lesssim (o_R(1)+\delta_1)d(u(t)),
\end{align*}
where again $o_R(1)$ represents norms of $Q_a$ at large radii that vanish as $R\to\infty$.

Collecting the estimates above, we obtain
\begin{align*}
|\eqref{E:UncVirialMain1}|+|\eqref{E:UncVirialMain2}|\lesssim \bigl(\tfrac{1}{\delta_1R^{2}}+o_R(1)+\delta_1\bigr)d(u(t)).
\end{align*}
As $F_\infty[u]=-4d(u(t))$, we can then choose $\delta_1$ sufficiently small and $R=R(\delta_1,Q_a)$ sufficiently large to guarantee that
\begin{equation}\label{E:UncDerivBnd}
\tfrac{d}{dt}P_R[u(t)]\leq -cd(u(t))
\end{equation}
for some $c>0$. Integrating over $[t_1,t_2]$, we obtain \eqref{E:UncVir}.

It remains to prove \eqref{E:UncVir2}.  We begin by showing that
\begin{equation}\label{e:UncPPos}
P_R[u(t)]\geq 0\qtq{for all}t\geq 0.
\end{equation}
Suppose instead that $P_R[u(t_0)]<-\eta$ for some $\eta>0$ and $t_0\geq 0$, then by \eqref{E:UncDerivBnd} and Proposition~\ref{P:Coercivity} we can derive that $P_R[u(t)]<-\eta$ for all $t>t_0$. Using \eqref{NLSa}, this implies
\[
\int|u(t,x)|^2w_R\,dx\leq \int|u(t_0,x)|^2w_R\,dx-\eta(t-t_0),
\]
yielding a contradiction if $t$ is sufficiently large. 

We next show that
\begin{equation}\label{E:UncPRBnd}
P_R[u(t)]\lesssim Rd(u(t))\qtq{for all}t\geq 0.
\end{equation}
If $d(u(t))>\delta_1$ we simply estimate
\[
P_R[u(t)]\lesssim R\|u\|_{H^1}^2\lesssim R[d(u(t))+\|Q_a\|_{H_a^1}^2]\lesssim R\bigl(1+\tfrac{1}{\delta_1}\|Q_a\|_{H_a^1}^2\bigr)d(u(t)).
\]
If instead $d(u(t))\leq d_1$, then we write
$u(t)=e^{i\theta(t)}[(1+\alpha(t))Q_a+h(t)]$ and use the fact that $P_R[e^{i\theta(t)}Q_a]\equiv 0$ to obtain
\begin{align*}
P_R[u(t)]&=|P_R[u(t)]-P_R[e^{i\theta(t)}Q_a]|\\
&\lesssim R\int|(1+\alpha(t))(Q_a\nabla h(t)+\nabla Q_ah(t))|+|\alpha(t)Q_a\nabla Q_a|\,dx\\
&\lesssim R(\|Q_a\|_{H_a^1}^2 [d(u(t))]^2+d(u(t))\|Q_a\|_{H_a^1}^2)\\
&\lesssim R(\delta_1+1)d(u(t)).
\end{align*}

Now, as $t\mapsto P_R[u(t)]$ is positive and decreasing, we may define $\ell=\displaystyle\lim_{t\to\infty}P_R[u(t)]$. As \eqref{E:UncVir} implies 
\[
\int_0^\infty d(u(t))\,dt<\infty,
\]
we must have $\ell=0$. Combining this fact with \eqref{E:UncVir} and \eqref{E:UncPRBnd} we obtain
\[
\int_t^\infty d(u(s))\,ds\lesssim d(u(t)),
\]
yielding \eqref{E:UncVir2} and completing the proof. \end{proof}

We now have all of the ingredients needed to prove Proposition~\ref{P:UncSolConv}.

\begin{proof}[Proof of Proposition~\ref{P:UncSolConv}]
By Lemma~\ref{L:UncVir} and Gronwall's inequality,
\[
\int_t^\infty d(u(s))\,ds\lesssim e^{-ct},
\]
for some $c>0$ and all $t$ large. Combining this Corollary~\ref{convofdistance}, we obtain
\[
\|u(t)-e^{i(t+\theta_0)}Q_a\|_{H_a^1}\lesssim e^{-ct},
\]
for some $\theta_0$ and $t$ sufficiently large, which implies the result.\end{proof}

\begin{cor}\label{C:FiniteVir}
If $u$ is as in Proposition~\ref{P:UncSolConv}, then in $xu(t)\in L^2$ for all $t\geq 0$.
\end{cor}
\begin{proof}
Define
\[
V_R[u(t)]=\int|u(t,x)|^2w_R\,dx.
\]
We proved above that for $R$ sufficiently large that
\[
\tfrac{d}{dt}V_R[u(t)]=P_R[u(t)]\geq 0,
\]
so that $t\mapsto V_R[u(t)]$ is nondecreasing. As 
\[
\|u(t)-e^{i(t+\theta)}Q_a\|_{H_a^1}\to 0\qtq{as}t\to\infty,
\]
we have
\[
V_R[u(t)]\leq V_R[Q_a]\leq \int|x Q_a|^2\,dx\qtq{for all}t\geq 0. 
\]
In particular, $V_R[u(0)]$ is uniformly bounded in $R$. Thus we can take the limit as $R\to\infty$ to deduce that $xu(0)\in L^2$, which implies the result. 
\end{proof}

%

\section{Proof of the main result}\label{S:Proof}

We are finally in a position to prove our main result, Theorem~\ref{T}. 

\begin{proof}[Existence of heteroclinic orbits] Recalling Proposition~\ref{spe}, we define $Q_a^\pm=U^{\pm1}$. By construction,
\[
\|Q^\pm\|_{\dot{H}_a^1}^2=\|Q_a\|_{\dot{H}_a^1}^2\pm2e^{-e_0t}\|Y_+\|_{\dot{H}_a^1}^2+\mathcal{O}(e^{-2e_0t}).
\]
In particular $\|Q_a^+(t)\|_{\dot{H}_a^1}>\|Q_a\|_{\dot{H}_a^1}$ and $\|Q_a^-(t)\|_{\dot{H}_a^1}<\|Q_a\|_{\dot{H}_a^1}$ for $t$ large.

Now, if $Q_a^-$ does not scatter backward in time, then Proposition~\ref{P:Precompactness} implies that the orbit $\{Q_a^-(t):t\in\R\}$ is precompact in $H^1$. But then by time reversal symmetry, Lemma~\ref{virialscat}, Corollary~\ref{finiteint}, and Corollary~\ref{convofdistance}, we have
\[
\int_{\R}d(Q_a^-(t))\,dt\lesssim \lim_{t\to\infty}d(Q_a^-(t))+\lim_{t\to-\infty}d(Q_a^-(t))=0.
\]
This implies that $d(Q_a^-(t))\equiv0$, contradicting that $\|Q_a^-(t)\|_{\dot{H}_a^1}<\|Q_a\|_{\dot{H}_a^1}$.

Next, we consider $Q_a^+$. As $Q_a^+$ is radial, Corollary~\ref{C:FiniteVir} implies $xQ_a^+\in L^2$. Moreover, by \eqref{e:UncPPos} we have that
\[
P_\infty[Q_a^+(0)]=2\Im\int Q_a^+(0)[x\cdot\overline{\nabla {Q_a^+}(0)}]\,dx\geq 0.
\]
In fact, since $\frac{d}{dt}P_\infty[Q_a^+(t)]<0$ we have that the quantity above is strictly positive. 

Now, if $Q_a^+$ fails to blow up in negative time, then we may apply the same argument to the time-reversed solution $\overline{Q_a^+}(t)=Q_a^+(-t)$ to obtain
\[
2\Im \int\overline{Q_a^+(0)}[x\cdot\nabla Q_a^+(0)]\,dx>0,
\]
a contradiction. \end{proof}

Finally, we prove the classification statement in Theorem~\ref{T}.

\begin{proof}[Classification of threshold solutions]  Let $u_0\in H_{\text{rad}}^1$ satisfy the mass-energy constraint
\[
M(u_0)E_a(u_0)=M(Q_a)E_a(Q_a)
\]
Using the scaling symmetry, we may in fact assume that
\[
M(u_0)=M(Q_a) \qtq{and} E_a(u_0)=E_a(Q_a). 
\]

First suppose $\|u_0\|_{\dot{H}_a^1}<\|Q_a\|_{\dot{H}_a^1}$.  If the corresponding solution $u$ does not scatter in both time directions, then by Proposition~\ref{P:Precompactness},  Proposition~\ref{compact_constrained}, Corollary~\ref{C:YouAreSpecial}, and Corollary~\ref{C:YouAreMoreSpecial}, we derive that $u=Q_a^-$ up to the symmetries of the equation.

Next suppose $\|u_0\|_{\dot{H}_a^1}>\|Q_a\|_{\dot{H}_a^1}$. If the corresponding solution $u$ does not blow up forward (say) in time, then by Proposition~\ref{P:UncSolConv}, Corollary~\ref{C:YouAreSpecial}, Corollary~\ref{C:YouAreMoreSpecial}, we derive that $u=Q_a^+$ up to the symmetries of the equation.  This completes the proof. \end{proof}

\bibliographystyle{plain}
\bibliography{biblio}
\nocite{*}
\end{document}